\documentclass[12pt]{amsart}
\usepackage{a4wide,amsmath,amssymb}
\usepackage[colorlinks=true]{hyperref}
\hypersetup{urlcolor=blue, citecolor=red}
\usepackage{graphicx}
\usepackage{subfigure}
\usepackage{color}

\newtheorem{theorem}{Theorem}
\newtheorem{remark}[theorem]{Remark}
\newtheorem{proposition}[theorem]{Proposition}
\newtheorem{lemma}[theorem]{Lemma}

\let\pa\partial
\let\na\nabla
\let\r\rho
\let\eps\varepsilon
\newcommand{\R}{\mathbb{R}}
\newcommand{\N}{\mathbb{N}}

\newcommand{\diver}{\textnormal{div}}

\begin{document}
\title[Preventing blow up in the Keller-Segel model]{Cross diffusion and
nonlinear diffusion preventing blow up in the Keller-Segel model}

\author[J. A. Carrillo]{Jos\'e Antonio Carrillo}
\address{ICREA and Departament de Matem\`atiques, Universitat Aut\`onoma de Barcelona,
08193 Bellaterra (Barcelona), Spain. {\it On leave from:}
Department of Mathematics, Imperial College London, London SW7 2AZ, UK.}
\author[S. Hittmeir]{Sabine Hittmeir}
\address{Department of Applied Mathematics and Theoretical Physics, Wilberforce Road,
Cambridge CB3 0WA, UK}
\author[A. J\"ungel]{Ansgar J\"ungel}
\address{Institute for Analysis and
Scientific Computing, Vienna University of Technology,
Wiedner Hauptstr.~8-10, 1040 Wien, Austria}

\thanks{The authors have been partially supported by the bilateral
Austrian-Spanish Project ES 08/2010-AT2009-0008 of the Austrian
Exchange Service (\"OAD) and MICINN. The work of SH is supported
by the King Abdullah University of Science and Technology (KAUST),
grant KUK-I1-007-43. SH and AJ acknowledge partial support from
the Austrian Science Fund (FWF), grants P20214, P22108, and I395,
from the Austrian-French Project FR 07/2010, and from the
Austrian-Croatian Project HR 01/2010 of the \"OAD. JAC was
partially supported by the Ministerio de Ciencia e Innovaci\'on,
grant MTM2011-27739-C04-02, and by the Ag\`encia de Gesti\'o
d'Ajuts Universitaris i de Recerca-Generalitat de Catalunya, grant
2009-SGR-345.}

\begin{abstract}
A parabolic-parabolic (Patlak-) Keller-Segel model in up to three
space dimensions with nonlinear cell diffusion and an additional
nonlinear cross-diffusion term is analyzed. The main feature of
this model is that there exists a new entropy functional, yielding
gradient estimates for the cell density and chemical
concentration. For arbitrarily small cross-diffusion coefficients
and for suitable exponents of the nonlinear diffusion terms, the
global-in-time existence of weak solutions is proved, thus
preventing finite-time blow up of the cell density. The global
existence result also holds for linear and fast diffusion of the
cell density in a certain parameter range in three dimensions.
Furthermore, we show
$L^\infty$ bounds for the solutions to the parabolic-elliptic system.
Sufficient conditions leading to the asymptotic stability
of the constant steady state are given for a particular choice of
the nonlinear diffusion exponents. Numerical experiments in two and three
space dimensions illustrate the theoretical results.
\end{abstract}



\maketitle


\section{Introduction}

Patlak \cite{Pat53} and Keller and Segel \cite{KeSe70} have proposed
a partial differential equation model,
which describes the movement of cells in response to a
chemical signal. The cells move towards regions of higher signal concentrations.
As the cells produce the signal substance, the movement may
lead to an aggregation of cells. The more cells are aggregated, the more the
attracting chemical signal is produced by the cells. This process is counter-balanced
by cell diffusion, but if the cell density is sufficiently large, the nonlocal chemical
interaction dominates and results -- in two and three space dimensions -- in a blow
up of the cell density (see the reviews \cite{HiPa09,Hor03} for details). Denoting by
$\rho=\rho(x,t)$ the cell density and by $c=c(x,t)$ the concentration of the
chemical signal, the Keller-Segel model, in its general form, is given by
\begin{align*}
  \pa_t\rho &= \diver(D(\rho)\na\rho - \chi(\rho)\na c) + R_1(\rho,c), \\
  \alpha \pa_t c &= \Delta c + R_2(\rho,S), \quad x\in\Omega,\ t>0,
\end{align*}
where $\Omega\subset\R^d$ ($d\ge 1$) is a bounded domain,
$D(\rho)$ is the cell diffusivity, $\chi(\rho)$ the chemotactic
sensitivity, and $R_1(\rho,c)$ and $R_2(\rho,c)$ describe the
production and degradation of the cell density and chemical
substance, respectively. Here, $\alpha=0$ corresponds to
the parabolic-elliptic case and $\alpha=1$ to the fully parabolic
problem. The equations are supplemented by homogeneous
Neumann boundary and initial conditions:
\begin{align*}
  D(\rho)(\na\rho \cdot\nu) = \na c\cdot\nu = 0 &\quad\mbox{on }\pa\Omega,\ t>0, \\
  \rho(\cdot,0) = \rho_0, \quad \alpha c(\cdot,0) = \alpha c_0 &\quad\mbox{in }\Omega,
\end{align*}
where $\nu$ denotes the exterior unit normal to the boundary
$\pa\Omega$ (which is assumed to exist).
The classical Keller-Segel model consists in the
choice $D(\rho)=1$, $\chi(\rho)=\rho$, $R_1(\rho,c)=0$, and
$R_2(\rho,c)=\rho-c$.

Motivated by numerical and modeling issues, the question how blow
up of cells can be avoided has been investigated intensively the
last years. Up to our knowledge, four methods have been proposed.
In the following, we review these methods.

The first idea is to modify the chemotactic sensitivity. Supposing
that aggregation stops when the cell density reaches the maximal
value $\rho_\infty=1$, one may choose $\chi(\rho)=1-\rho$. In this
volume-filling case, the cell density is bounded, $0\le\rho\le 1$,
and the global existence of solutions can be proved \cite{DiRo08}.
Furthermore, if $\chi(\rho)=\rho^\beta$ with $0<\beta<2/d$, the
solutions are global and bounded, thus preventing finite-time blow
up \cite{HoWi05}. Global solutions are also obtained when the
sensitivity depends on the chemical concentrations in an
appropriate way, see, e.g., \cite{Bil99,HPS07}.

A second method consists in modifying the cell diffusion. In the
context of the volume-filling effect, Burger et al.\ \cite{BDD06}
suggested the cell equation $\pa_t\rho =
\diver(\rho(1-\rho)\na(\rho-c))$. Then the parabolic-elliptic
model possesses global solutions. Global existence results can be
achieved by employing the nonlinear diffusion
$D(\rho)=\rho^\alpha$, which models the local repulsion of cells.
When $\int_1^\rho (D(s)/s)ds$ grows faster than $\log\rho$ for
large $\rho$, a priori estimates showing that solutions are global
and uniformly bounded in time were obtained in
\cite{CaCa06,Kow05}. Adding the nonlinear sensitivity
$\chi(\rho)=\rho^{\beta}$ with $\alpha\ge 1$ and $2\le
\beta<\alpha+2/d$, global existence results were achieved in
\cite{IsYo11}. The solutions are uniformly bounded in time if
$\alpha>2-4/d$ \cite{KoSz08}. The existence of global bounded
classical solutions to a fast-diffusion Keller-Segel model with
$D(\rho)=(1-\rho)^{-\alpha}$, where $\alpha\ge 2$, has been proved
in \cite{ChWa10}. The same result holds true when we choose
$\chi(\rho)=\rho(1-\rho)^\beta$ with $\beta\ge 1-\alpha/2$, and
the solution is still global in time (but possibly not classical)
if $\beta\ge 1-\alpha$ \cite{Wrz10}.

A third approach is to consider nonvanishing growth-death models $R_1\neq 0$,
since one may expect that a suitable death term avoids cell aggregation.
Indeed, taking $R_1(\rho,c)=\rho(1-\rho)(\rho-a)$ for some $0\le a\le 1$, the global
existence of solutions is proved in \cite{BHLM96}. In the logistic-growth
model $R_1(\rho,c)=\rho(1-\rho^{\gamma-1})$, a global weak solution exists
for all $\gamma>2-1/d$ \cite{Win08}. These results have been obtained
for the parabolic-elliptic model.

Recently, a fourth way to obtain global existence of solutions has been proposed
\cite{HiJu11}. The idea is to add a cross-diffusion term
in the equation for the chemical signal:
\begin{align*}
  \pa_t\rho &= \diver(\na\rho - \rho\na c), \\
  \alpha \pa_t c &= \Delta c + \delta \Delta\rho + \rho - c \quad\mbox{in }\Omega, \ t>0,
\end{align*}
where $\delta>0$. At first sight, the additional cross-diffusion term seems to
cause several mathematical difficulties since the diffusion matrix of the
above system is neither symmetric nor positive definite, and we cannot
apply the maximum principle to the equation for the chemical signal anymore.
All these difficulties can be resolved by the observation that the above system
possesses a logarithmic entropy,
$$
  E_0(\rho,c) = \int_\Omega \left[\rho(\log\rho-1) + \alpha \frac{c^2}{2\delta}\right]dx,
$$
allowing for global existence
results and revealing some interesting structural properties of the system.
In fact, the entropy production equation
$$
  \frac{dE_0}{dt} + \int_\Omega\Big(4|\na\sqrt{\rho}|^2 + \frac{1}{\delta}|\na c|^2
  + \frac{1}{\delta}c^2\Big)dx = \frac{1}{\delta}\int_\Omega \rho c dx
$$
and suitable Gagliardo-Nirenberg estimates for the right-hand side
lead to gradient estimates for $\sqrt{\rho}$ and $c$. Another motivation for
the introduction of the additional cross diffusion is that,
whereas finite-element discretizations of the classical Keller-Segel model break
down some time before the blow up, the numerical solutions to the augmented model
exists for all time, which may lead to estimates of the blow-up time. This
question is currently under investigation.

In \cite{HiJu11}, the existence of global weak solutions has been
proved in the two-dimen\-sio\-nal situation only. In this paper,
we generalize this result to three space dimensions by allowing
for nonlinearities in the cell diffusion terms. Since nonlinear
diffusion in the cell equation helps to achieve global existence
results (see above), we suggest, in contrast to \cite{HiJu11}, a
{\em nonlinear} cross-diffusion term. More precisely, we consider
the equations
\begin{align}
  \pa_t\rho &= \diver(\na(\rho^m) - \rho\na c), \label{1.rho} \\
  \alpha \pa_t c &= \Delta c + \delta\Delta(\rho^n) + \rho - c
  \quad\mbox{in }\Omega,\ t>0,
  \label{1.c}
\end{align}
subject to the no-flux and initial conditions
\begin{align}
  (\na(\rho^m)-\rho\na c)\cdot\nu
  = \na(c+\delta\rho^n)\cdot\nu = 0 &\quad\mbox{on }\pa\Omega,\ t>0, \label{1.bc} \\
  \rho(\cdot,0) = \rho_0, \ \alpha c(\cdot,0) = \alpha c_0 &\quad\mbox{in }\Omega. \label{1.ic}
\end{align}
Notice that these boundary conditions are equivalent to
$\nabla\rho\cdot\nu=\nabla c\cdot \nu=0$ on $\pa\Omega$
for smooth positive solutions.

In two space dimensions, the case $m=n=1$ is covered by
\cite{HiJu11}. If $m>3-4/d$, $2\leq d\leq 3$, the nonlinear
diffusion already prevents blow-up of the solutions without
additional cross diffusion, see \cite{KimYao,Kow05,KoSz08}. The question
remains if we can allow for linear and fast diffusion of cells,
$m\leq 1$, for some $n>1$, and still obtain global existence
results. In this paper, we show that this is indeed true. For
instance, we show that in the presence of the additional cross
diffusion term and in three space dimensions, we can allow for the
classical cell diffusion exponent $m=1$ and still obtain global
existence results. This shows that the result of \cite{HiJu11} can
be generalized to the three-dimensional case if the cross
diffusion is of degenerate type. These remarks motivate us to
restrict ourselves to the case $m>0$ and $n>1$. Our first main result is
as follows.

\begin{theorem}[Global existence]\label{thm.ex}
Let $\Omega\subset\R^d$ $(1\leq d\le 3)$ be a bounded domain with
$\pa\Omega\in C^{1,1}$. Let $\alpha\ge 0$, $m>0$, $n>1$, and let $p=(m+n-1)/2$
satisfy $1-n/d<p\le\min\{m,n\}$. Furthermore, let
$0\le\rho_0\in L^n(\Omega)$ and $\alpha c_0\in L^2(\Omega)$. Then
there exists a global weak solution $(\rho,c)$ to
\eqref{1.rho}-\eqref{1.ic} satisfying $\rho\ge 0$ in $\Omega$,
$t>0$, and, for some $s\in(1,2]$,
\begin{align*}
  & \rho\in L^\infty_{\rm loc}(0,\infty;L^n(\Omega))\cap
  L^{2Q}_{\rm loc}(0,\infty;L^{2Q}(\Omega)), \\
  & \rho^m,\ \rho^n\in L^s_{\rm loc}(0,\infty;W^{1,s}(\Omega)), \quad
  \rho\na c\in L^s_{\rm loc}(0,\infty;L^s(\Omega)), \\
  & \alpha c\in L^\infty_{\rm loc}(0,\infty;L^2(\Omega)), \quad
  c\in L^2_{\rm loc}(0,\infty;H^1(\Omega)), \\
  & \pa_t\rho,\ \alpha \pa_t c\in L^s_{\rm loc}(0,\infty;(W^{1,s}(\Omega))'),
\end{align*}
where $Q=n/d+p>1$.
\end{theorem}
\begin{remark}\rm\label{rem.weak}
A weak solution is to be understood in the standard sense by testing
the system of equations against compactly supported smooth functions in
$C^\infty((0,T)\times \Omega))$. Due to the regularity properties
of the solution, however, test functions in
$L^s(0,T;W^{1,s}(\Omega))$ are sufficient for the weak formulation
of both equations in the fully parabolic system to be well
defined. For the parabolic-elliptic system, we show in Section
\ref{sec.unif} that we can even allow for test functions in
$L^2(0,T;H^1(\Omega))$.
\qed
\end{remark}

Let us discuss the conditions on $p$ which are equivalent to
\begin{equation}\label{conditions}
  m-1\le n\le m+1,  \quad m+n+\frac{2}{d}n > 3.
\end{equation}
The areas of admissible values for $(m,n)$ are illustrated in Figure
\ref{fig.d}. Notice that the bands between $n-1\leq m\leq
n+1$ continue to the right.

\begin{figure}[ht]
\centering
\includegraphics[width=52mm,height=40mm]{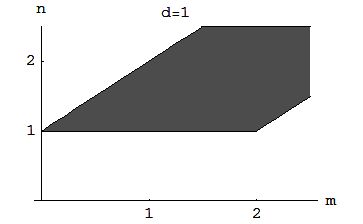}
\includegraphics[width=52mm,height=40mm]{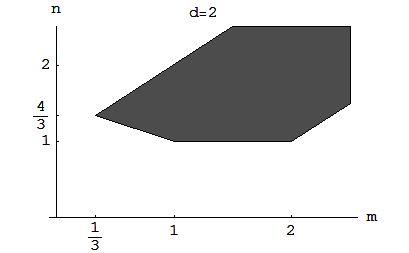}
\includegraphics[width=52mm,height=40mm]{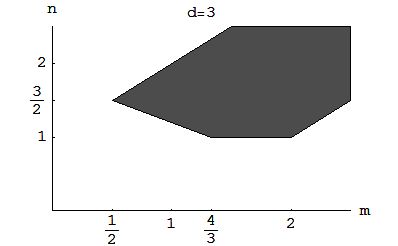}
\caption{Admissible values $(m,n)$ for $d=1$ (left), $d=2$ (middle),
and $d=3$ (right).}
\label{fig.d}
\end{figure}

In the fast-diffusion case, for $d=2$, we may take $\frac13<m<1$
and $\frac12(3-m)<n\le m+1$; for $d=3$, the values $\frac12<m<1$
and $\frac35(3-m)<n\le m+1$ are admissible. For classical
diffusion, $m=1$, the above conditions are satisfied for any
$1<n\le 2$ (if $d=2$) and $6/5<n\le 2$ (if $d=3$). Hence, the
degenerate cross-diffusion term prevents blow up in finite time
even in the case of linear cell diffusion in three dimensions. In
short, one of the conditions in \eqref{conditions} is needed to
derive a nice bound on an entropy functional and the others for
suitable compactness and continuity properties of the approximated sequences.

The key idea of the proof of Theorem \ref{thm.ex} is the observation that system
\eqref{1.rho}-\eqref{1.c} possesses an entropy functional,
\begin{equation}\label{1.ent}
  E(\rho,c) = \int_\Omega\Big(\frac{\rho^n}{n-1} +
  \alpha \frac{c^2}{2\delta}\Big)dx,
\end{equation}
useful to derive a priori estimates. Indeed, differentiating
formally this functional, we obtain the entropy production
equation
$$
  \frac{dE}{dt} + \int_\Omega\Big(\frac{mn}{p^2}|\na \rho^p|^2
  + \frac{1}{\delta}|\na c|^2 + \frac{c^2}{\delta}\Big)dx
  = \frac{1}{\delta}\int_\Omega\rho cdx,
$$
recalling that $p = (m+n-1)/2$. We will show in the proof of Lemma \ref{lem.ent}
that the right-hand side can be estimated for any $\beta>0$ as follows:
\begin{equation}\label{1.aux}
  \int_\Omega\rho cdx \le \beta\int_\Omega|\na\rho^p|^2 dx
  + C(\beta,\|\rho\|_{L^1(\Omega)}) + \frac{1}{2}\int_\Omega(|\na c|^2 + c^2)dx,
\end{equation}
under the restriction $1-2/d<p$, which follows from the conditions
in \eqref{conditions} for $1\leq d\leq 3$. The assumptions $m>0$
and $n>1$ imply that $p>0$, thus only for $d=3$ we obtain the
restriction $p>1/3$. Let us remark that in the case $d=3$, the
restriction $1-2/d<p$ is redundant, that is, conditions
\eqref{conditions} together with $n>1$ and $m>0$ imply that
$p+2/d\geq p+n/d>1$ for $n\in(1,2]$, as well as $p>1$ since
$2p=m+n-1> 2$ for $n>2$ with $m\geq n-1\geq 1$.

The existence of the entropy functional \eqref{1.ent} implies the existence
of so-called entropy variables which makes the new diffusion matrix positive
(semi-) definite. Indeed, introducing the entropy variables
$$
  r = \frac{\pa E}{\pa\rho} = \frac{n}{n-1}\rho^{n-1}, \quad
  b = \frac{\pa E}{\pa c} = \frac{c}{\delta},
$$
system \eqref{1.rho}-\eqref{1.c} can be written as
\begin{equation}\label{1.rho.c}
  \frac{\pa}{\pa t}\begin{pmatrix} \rho \\ \alpha c \end{pmatrix}
  - \diver\left(\begin{pmatrix}
  (m/n)\rho^{m-n+1} & -\delta\rho \\ \delta\rho & \delta
  \end{pmatrix}\na\begin{pmatrix} r \\ b \end{pmatrix}\right)
  = \begin{pmatrix} 0 \\ \rho-c \end{pmatrix}.
\end{equation}
In hyperbolic or parabolic systems, the existence of an entropy
functional is equivalent to the existence of a change of unknowns which ``symmetrizes''
the system \cite{DGJ97,KaSh88}. (For parabolic systems, ``symmetrization'' means that
the transformed diffusion matrix is symmetric and positive definite.)
In system \eqref{1.rho.c}, the diffusion matrix is
nonsymmetric, but still positive semi-definite.

The existence proof is based on the construction of a problem
which approximates \eqref{1.rho.c}. First, we replace the time
derivative by an implicit Euler approximation with time step
$\tau>0$ and add a weak form of the fourth-order operator
$\eps(\Delta^2 r-\diver(|\nabla r|^2 \nabla r)+r)$ ($\eps>0$)
to the first component of \eqref{1.rho.c}, which guarantees the
coercivity of the elliptic system in $H^2(\Omega)$ with respect to
$r$. The existence of weak approximating solutions
$(r_\eps,b_\eps)$ is shown by using the Leray-Schauder fixed-point
theorem. At this point, we need the restriction $p\leq m$,
which is equivalent to $n\leq m+1$, in \eqref{conditions}
to ensure the continuity and coercivity.
The discrete entropy estimate implies a priori estimates
uniform in the approximation parameters $\tau$ and $\eps$, which
allow us to pass to the limit $(\tau,\eps)\to 0$.

There are two technical difficulties in the limiting procedure.
The first one is that the entropy estimate yields a uniform bound
for $\rho_\eps^p$ in $H^1(\Omega)$, but an estimate for
$\pa_t\rho_\eps$ in $(H^{3}(\Omega))'$. If $p\le 1$, this implies
a bound for $\rho_\eps$ in $W^{1,r}(\Omega)$ for some $r>0$, and
we can apply the Aubin lemma to conclude the relative compactness
of the family $(\rho_\eps)_{\eps>0}$. If $p>1$, we infer this
property using a variant of the Dubinskii lemma
(see Lemma \ref{lem.comp}). The
second difficulty is to ensure the strong convergence of the
family $(\rho_\eps)_{\eps>0}$ in $L^2(\Omega\times(0,T))$. In two
space dimensions $d=2$ (and with $n=m=1$), this has been proved in
\cite{HiJu11}. However, for $d=3$ (and $p<1$), our uniform
estimates in Lemma \ref{lem.est.s} need additional assumptions on
the diffusion parameters, namely $p>1-n/d$ and $p\leq n$, or
equivalently, the remaining two conditions in \eqref{conditions}:
$m+n+2n/d>3$ and $m-1\leq n$.

Our second main result concerns some qualitative properties of the solutions
to \eqref{1.rho}-\eqref{1.ic} using the entropy functional.
First, we prove $L^\infty$ bounds for the solutions to the
parabolic-elliptic system generalizing the results of \cite{Kow05,KoSz08}
to this situation.

\begin{theorem}[Boundedness in $L^\infty$]\label{thm.bound}
Let the assumptions of Theorem \ref{thm.ex} hold and let $\alpha=0$.
Then, for any $T>0$, the solution $(\rho,c)$ to the parabolic-elliptic system
\eqref{1.rho}-\eqref{1.ic} satisfies
$$
  \|\rho\|_{L^\infty(0,T;L^\infty(\Omega))}
  + \|c\|_{L^\infty(0,T;L^\infty(\Omega))} \le C(T),
$$
where the constant $C(T)>0$ depends on $T>0$.
\end{theorem}

Second, we are able to show the asymptotic stability of solutions
to the constant steady state. Due to the special structure of the
entropy functional, we can allow for a very particular choice of the
parameters $m$, $n$, and $\delta$ only.

\begin{proposition}[Long-time decay for $m=1$, $n=2$]\label{prop.decay}
Let $\Omega\subset \mathbb{R}^d$ $(1\leq d\leq3)$ be a bounded
domain with $\partial \Omega\in C^{1,1}$. Let $\r_0 \in
L^\infty(\Omega)$, $m=1$, $n=2$, and $\delta>C_P^2/4$, where
$C_P$ is the constant of the Poincare inequality in $L^2(\Omega)$.
Then the solution to the parabolic-elliptic system
\eqref{decay.rho}-\eqref{decay.c} with $\alpha=0$, constructed in
Theorem \ref{par-ell-thm}, decays exponentially fast to the
homogeneous steady state in the sense that
\[
  \|\r(\cdot,t)-\r^*\|_{L^2(\Omega)}\leq Ce^{-\kappa t}, \quad
  \|c(\cdot,t)-c^*\|_{L^1(\Omega)}\leq Ce^{-\kappa t},
\]
where $C>0$ is some constant and $\kappa=\min\{1,4\delta-C_P^2\}/(4\delta)$.
Moreover, any smooth solution $(\r,c)$ to the fully parabolic
system \eqref{decay.rho}-\eqref{decay.c} with $\alpha=1$ has the
decay properties
\[
  \|\r(\cdot,t)-\r^*\|_{L^2(\Omega)}\leq Ce^{-\kappa t}, \quad
  \|c(\cdot,t)-c^*\|_{L^2(\Omega)}\leq Ce^{-\kappa t}
\]
for all $t>0$.
\end{proposition}

The paper is organized as follows. In Section \ref{sec.aux}, we
prove an inequality which is needed for the proof of \eqref{1.aux}
and we show a compactness result which combines the lemmas of
Aubin and Dubinskii. Theorems \ref{thm.ex} and \ref{thm.bound} are
shown in Sections \ref{sec.ex} and \ref{sec.unif}, respectively,
whereas Proposition \ref{prop.decay} is proved in Section \ref{sec.lt}.
Finally, in Section \ref{sec.numer} we present some numerical results in
two and three space dimensions which illustrate the effect of the
exponent $n$.


\section{Auxiliary results}\label{sec.aux}

\begin{lemma}\label{lem.ineq}
Let $\Omega\subset\R^d$ $(d\ge 1)$ be a bounded domain with
$\pa\Omega\in C^{0,1}$.
Furthermore, let $\beta>0$ and
\begin{itemize}
\item
  \begin{tabbing}
  if $d\geq 3$: \= either $1-2/d<p\le 1$ and
  $q=2d/(d+2)$ \\
  \> or $p>1$ and $q=p+1/2$,
  \end{tabbing}
\item
  \begin{tabbing}
  if $d\leq 2$: \= either $0<p\le 1$, $q>1$, and $p+1/q>3/2-1/d$ \\
  \> or $p>1$ and $q=p+1/2$.
  \end{tabbing}
\end{itemize}
Then there exists a constant $C(\beta,\|\r\|_{L^1(\Omega)})>0$ such that for all
$\rho\in L^1(\Omega)$ satisfying $\rho\ge 0$ in $\Omega$ and $\rho^p\in H^1(\Omega)$,
the following inequality holds:
$$
  \|\rho\|_{L^q(\Omega)}^2
  \le \beta\|\na\rho^p\|_{L^2(\Omega)}^2 + C(\beta,\|\rho\|_{L^1(\Omega)}).
$$
\end{lemma}

Notice that the continuous embedding $H^1(\Omega)\hookrightarrow L^s(\Omega)$,
where $1\le s\le 2d/(d-2)$ if $d\ge 3$ and $1\le s<\infty$ if $d\le 2$, shows
that $\rho^p\in H^1(\Omega)$ implies that $\rho\in L^{sp}(\Omega)$, and the
condition $q\le sp$ has to be imposed. This condition is satisfied for the
above choices of $p$ and $q$.

\begin{proof}
First, let $0<p\le 1$ and $(p,q)$ be given as in the Lemma.
The Gagliardo-Nirenberg inequality, see e.g. \cite[Theorem 10.1]{Fri69} and \cite[Theorem 1.1.4]{Zhe95},
 gives
\begin{align*}
  \|\rho\|_{L^q(\Omega)}^2
  &= \|\rho^p\|_{L^{q/p}(\Omega)}^{2/p}
  \le C\|\na\rho^p\|_{L^2(\Omega)}^{2\theta/p}
  \|\rho^p\|_{L^{1/p}(\Omega)}^{2(1-\theta)/p} + C\|\rho^p\|_{L^{1/p}(\Omega)}^{2/p} \\
  &= C\|\na\rho^p\|_{L^2(\Omega)}^{2\theta/p} \|\rho\|_{L^1(\Omega)}^{2(1-\theta)}
  + C\|\rho\|_{L^1(\Omega)}^2,
\end{align*}
where $\theta=dp(1-1/q)/(1-d/2+dp)$ and $C>0$ is here and in the following a
generic constant. The conditions $p>1-2/d$ if $d\ge 3$ and
$q>1$ if $d\le 2$ imply that $\theta>0$. For all space dimensions, it holds that
$p+1/q>3/2-1/d$ which is equivalent to $\theta<p\le 1$. Then the inequality
$\theta/p<1$ allows us to apply the Young inequality:
$$
  \|\rho\|_{L^q(\Omega)}^2 \le \beta\|\na\rho^p\|_{L^2(\Omega)}^2
  + C(\beta) \|\rho\|_{L^1(\Omega)}^{2p(1-\theta)/(p-\theta)}
  + C\|\rho\|_{L^1(\Omega)}^2,
$$
proving the first case.

Next, let $p>1$ and $q=p+1/2$. Notice that the Poincar\'e inequality implies that
$$
\|f\|_{L^2(\Omega)}\le \left\|f-\int_{\Omega} f dx\right\|_{L^2(\Omega)}+ \|f\|_{L^1(\Omega)}
  \le C_P \|\na f\|_{L^2(\Omega)}+ \|f\|_{L^1(\Omega)}.
$$
This together with the H\"older inequality leads to
\begin{align*}
  \|\rho\|_{L^q(\Omega)}^2
  &= \|\rho^q\|_{L^{1}(\Omega)}^{2/q}
  = \|\rho^p\rho^{1/2}\|_{L^{1}(\Omega)}^{2/q} \\
  &\le \|\rho^p\|_{L^2(\Omega)}^{2/q}\|\rho^{1/2}\|_{L^2(\Omega)}^{2/q}
  \le C\big(\|\na\rho^p\|_{L^2(\Omega)}^{2/q} + \|\rho^p\|_{L^1(\Omega)}^{2/q}\big)
  \|\rho^{1/2}\|_{L^2(\Omega)}^{2/q} \\
  &= C\|\na\rho^p\|_{L^2(\Omega)}^{2/q}\|\rho\|_{L^1(\Omega)}^{1/q}
  + C\|\rho\|_{L^p(\Omega)}^{2p/q}\|\rho\|_{L^1(\Omega)}^{1/q}.
\end{align*}
Furthermore, using the interpolation inequality with $1/p = \theta/q + (1-\theta)/1$
or, equivalently, $p\theta/q=(p-1)/(q-1)>0$,
$$
  \|\rho\|_{L^q(\Omega)}^2
  \le C\|\na\rho^p\|_{L^2(\Omega)}^{2/q}\|\rho\|_{L^1(\Omega)}^{1/q}
  + C\|\rho\|_{L^q(\Omega)}^{2p\theta/q}\|\rho\|_{L^1(\Omega)}^{2p(1-\theta)/q}
  \|\rho\|_{L^1(\Omega)}^{1/q}.
$$
Since $q>1$, we may employ the Young inequality for the first summand to obtain
$$
  \|\rho\|_{L^q(\Omega)}^2
  \le \beta\|\na\rho^p\|_{L^2(\Omega)}^2 + C(\beta)\|\rho\|_{L^1(\Omega)}^{1/(q-1)}
  + C\|\rho\|_{L^q(\Omega)}^{2p\theta/q}\|\rho\|_{L^1(\Omega)}^{2p(1-\theta)/q+1/q}.
$$
Since $1<p<q$, it follows that $2p\theta/q<2$, which
allows us to use the Young inequality for the second summand:
$$
  \|\rho\|_{L^q(\Omega)}^2
  \le \beta\|\na\rho^p\|_{L^2(\Omega)} + C(\beta)\|\rho\|_{L^1(\Omega)}^{1/(q-1)}
  + \frac12\|\rho\|_{L^q(\Omega)}^2 + C(\beta,\|\rho\|_{L^1(\Omega)}).
$$
The lemma is proved.
\end{proof}

Next, we recall a compactness result. Let $(\sigma_h\rho)(x,t)
=\rho(x,t-h)$, $t\ge h > 0$, be a shift operator.

\begin{lemma}[Dubinskii]\label{lem.dub}
Let $\Omega\subset\R^d$ $(d\ge 1)$ be a bounded domain with
$\pa\Omega\in C^{0,1}$ and let $T>0$. Furthermore, let
$p\ge 1$, $q\ge 1$, and $s\ge 0$,
and let $(\rho_\eps)$ be a sequence of nonnegative functions satisfying
$$
  h^{-1}\|\rho_\eps-\sigma_h\rho_\eps\|_{L^1(h,T;(H^{s}(\Omega))')}
  + \|\rho_\eps^p\|_{L^q(0,T;H^1(\Omega))} \le C \quad\mbox{for all }h>0,
$$
where $C>0$ is a constant which is independent of $\eps$ and $h$.
Then $(\rho_\eps)$ is relatively compact in $L^{p\ell}(0,T;L^{pr}(\Omega))$
for all $\ell<q$
and for all $r<2d/(d-2)$ if $d\ge 3$, $r<\infty$ if $d=2$, and $r\le \infty$ if $d=1$.
\end{lemma}

A variant of this lemma is due to Dubinskii, see
\cite[Th\'eor\`eme 21.1, Chapter~1]{Lio69} for a proof.
A simple proof is achieved by applying
the lemmas of Aubin \cite{Sim87} and Chavent-Jaffre \cite{ChJa86}.
Since the result is of interest by itself, we provide the (short) proof.

\begin{proof}
The function $f(x)=x^{1/p}$, $0<x<\infty$, is H\"older continuous with exponent
$1/p$. Therefore, with $u=\rho_\eps^p$,
by the lemma of Chavent-Jaffre \cite[p.~141]{ChJa86},
$$
  \|\rho_\eps\|_{W^{1/p,2p}(\Omega)} = \|f(u)\|_{W^{1/p,2p}(\Omega)}
  \le K\|u\|_{H^1(\Omega)}^{1/p} = C\|\rho_\eps^p\|_{H^1(\Omega)}^{1/p}.
$$
This shows that $(\rho_\eps)$ is bounded in $L^{pq}(0,T;W^{1/p,2p}(\Omega))$.
By Aubin's lemma \cite[Theorem 6]{Sim87} and the compact embedding
$W^{1/p,2p}(\Omega)\hookrightarrow L^{pr}(\Omega)$ ($r$ is as in the lemma),
$(\rho_\eps)$ is relatively compact in $L^{p\ell}(0,T;L^{pr}(\Omega))$ for all
$\ell<q$.
\end{proof}

The following result, which will be used in this paper,
is a consequence of Lemma \ref{lem.dub}.

\begin{lemma}\label{lem.comp}
Let $\Omega\subset\R^d$ $(d\ge 1)$ be a bounded domain with
$\pa\Omega\in C^{0,1}$, let $T>0$, $\tau>0$, and let $t_k=k\tau$, $k=0,\ldots,N$,
with $N\tau = T$ be a decomposition of the interval $[0,T]$.
Furthermore, let $p\ge 1$, $q\ge 1$, and $s\ge 0$,
and let $(\rho_\tau)$ be a sequence of nonnegative functions, which are
piecewise constant in time on $(0,T)$, satisfying
$$
  \tau^{-1}\|\rho_\tau-\sigma_\tau\rho_\tau\|_{L^1(\tau,T;(H^s(\Omega))')}
  + \|\rho_\tau^p\|_{L^q(0,T;H^1(\Omega))} \le C \quad\mbox{for all }\tau>0,
$$
where $C>0$ is a constant which is independent of $\tau$.
Then $(\rho_\tau)$ is relatively compact in $L^{p\ell}(0,T;L^{pr}(\Omega))$
for all $\ell<q$ and
for all $r<2d/(d-2)$ if $d\ge 3$, $r<\infty$ if $d=2$, and $r\le \infty$ if $d=1$.
\end{lemma}

\begin{proof}
Since $\rho_\tau$ is piecewise constant in time, we can write
$\rho_\tau(\cdot,t)=\rho_k$ for $t\in((k-1)\tau,k\tau]$, $k=1,\ldots,N$,
for some functions $\rho_k$.

\emph{Case $h<\tau$.}
The difference $\rho_\tau-\sigma_h\rho_\tau$ partially cancels for $h<\tau$, and
we obtain, for $k=1,\ldots,N-1$ and $t>h$,
$$
  \|\rho_\tau(\cdot,t)-(\sigma_h\rho_\tau)(\cdot,t)\|_{(H^s(\Omega))'}
  = \left\{\begin{array}{ll}
  \|\rho_{k+1}-\rho_k\|_{(H^s(\Omega))'} &\mbox{if }t_k<t\leq t_k+ h, \\
  0 &\mbox{else}.
  \end{array}\right.
$$
Therefore, by assumption,
\begin{align*}
  h^{-1} &\|\rho_\tau-\sigma_h\rho_\tau\|_{L^1(h,T;(H^s(\Omega))')}
  = h^{-1}\sum_{k=1}^{N-1}\int_{t_k}^{t_{k}+h}
  \|\rho_{k+1}-\rho_k\|_{(H^s(\Omega))'} dt \\
  &= \sum_{k=1}^{N-1}\|\rho_{k+1}-\rho_k\|_{(H^s(\Omega))'}
  = \tau^{-1}\sum_{k=1}^{N-1}\int_{t_k}^{t_{k+1}}
  \|\rho_\tau(\cdot,t)-(\sigma_\tau\rho_\tau)(\cdot,t)\|_{(H^s(\Omega))'} dt
  \le C
\end{align*}
uniformly in $h<\tau$.

\emph{Case $h\geq \tau$.} There exists $m\in \mathbb{N}$ such that
$t_m<h\leq t_{m+1}$. Then, for $t\in(t_{k+m-1},t_{k+m}]$, $k=1,\dots N-m$,
$$
  \|(\r_\tau-\sigma_h\r_\tau)(\cdot,t)\|_{(H^s(\Omega))'}
  =\left\{\begin{array}{ll}\|\r_{k+m}-\r_{k-1}\|_{(H^s(\Omega))'}
  & \textnormal{if} \ t_{k+m-1}<t\leq t_{k-1}+h \\
  \|\r_{k+m}-\r_{k}\|_{(H^s(\Omega))'}
  & \textnormal{if} \ t_{k-1}+h<t\leq t_{k+m}
  \end{array}\right.
$$
We compute
\begin{align*}
  \|\r_\tau & -\sigma_h\r_\tau\|_{L^1(h,T;(H^s(\Omega))')}\\
  &=\int_h^{t_{m+1}}\|\r_\tau-\sigma_h\r_\tau\|_{(H^s(\Omega))'}dt
  +\sum_{k=2}^{N-m}\int_{t_{k+m-1}}^{t_{k+m}}\|\r_\tau-\sigma_h
  \r_\tau\|_{(H^s(\Omega))'}dt\\
  &= \int_h^{t_{m+1}}\|\r_{m+1}-\r_1\|_{(H^s(\Omega))'}dt
  +\sum_{k=2}^{N-m}\int_{t_{k+m-1}}^{t_{k-1}+h}
  \|\r_{k+m}-\r_{k-1}\|_{(H^s(\Omega))'}dt  \\
  &\phantom{xx}{} +\sum_{k=2}^{N-m}\int_{t_{k-1}+h}^{t_{k+m}}
  \|\r_{k+m}-\r_{k}\|_{(H^s(\Omega))'}dt\\
  &=(t_{m+1}-h)\|\r_{m+1}-\r_1\|_{(H^s(\Omega))'}
  + (h-t_m)\sum_{k=2}^{N-m}\|\r_{k+m}-\r_{k-1}\|_{(H^s(\Omega))'} \\
  &\phantom{xx}{} +(t_{m+1}-h)\sum_{k=2}^{N-m}\|\r_{k+m}-\r_k\|_{(H^s(\Omega))'}.
\end{align*}
We employ the estimates $h-t_m\leq \tau$, $t_{m+1}-h\leq \tau$
and the triangle inequality:
\begin{align*}
  \|\r_\tau & -\sigma_h\r_\tau\|_{L^1(h,T;(H^s(\Omega))')}
  \leq \tau \sum_{j=2}^{m+1}\|\r_j-\r_{j-1}\|_{(H^s(\Omega))'}\\
  &\phantom{xx}{}+\tau\sum_{j=0}^m\sum_{k=2}^{N-m}
  \|\r_{k+j}-\r_{k+j-1}\|_{(H^s(\Omega))'}
  +\tau \sum_{j=0}^m\sum_{k=2}^{N-m}\|\r_{k+j}-\r_{k+j-1}\|_{(H^s(\Omega))'}.
\end{align*}
Since
\[
  \sum_{k=2}^{N-m}\sum_{j=i}^ma_{k+j}=\sum_{j=i}^m\sum_{\ell=2+j}^{N-m+j} a_\ell
  \leq (m-i+1)\sum_{\ell=2}^Na_\ell
  \]
for numbers $a_\ell\geq 0$ and all $0\leq i\leq m$, it follows that
\begin{align*}
  \|\r_\tau-\sigma_h\r_\tau\|_{L^1(h,T;(H^s(\Omega))')}
  &\leq 3(m+1)\tau \sum_{k=1}^{N-1}\|\r_{k+1}-\r_k\|_{(H^s(\Omega))'}\\
  &= 3(m+1)\int_\tau^T\|\r_\tau-\sigma_\tau\r_\tau\|_{(H^s(\Omega))'} dt.
\end{align*}
Thus, using $(m+1)\tau \leq h+\tau\leq 2h$,
\[
  \|\r_\tau-\sigma_h\r_\tau\|_{L^1(h,T;(H^s(\Omega))')}
  \leq \frac{6h}{\tau}\|\r_\tau-\sigma_\tau\r_\tau\|_{L^1(\tau,T;(H^s(\Omega))')}
  \leq Ch.
\]
We conclude that in both cases, for all $h>0$,
\[
  \|\r_\tau-\sigma_h\r_\tau\|_{L^1(h,T;(H^s(\Omega))')}\leq C h,
\]
and this estimate is uniform in $\tau>0$. We apply Lemma
\ref{lem.dub} to conclude the result.
\end{proof}


\section{Global existence of weak solutions}\label{sec.ex}

In this section, we prove Theorem \ref{thm.ex}. Let in the
following $\Omega\subset\mathbb{R}^{d}$ ($1\leq d\leq 3$) be a
bounded domain with $\partial\Omega \in C^{1,1}$. The smoothness
assumption on the boundary of the domain is needed for applying
elliptic regularity results.

\subsection{Solution of an approximate problem}

We show first the existence of a weak solution to an approximate problem
which is obtained by semi-discretizing \eqref{1.rho}-\eqref{1.c} with respect to time
and by regularizing the equation for the cell density. For this, let $T>0$ and
$K\in\N$ and split the time interval in the subintervals
$$
  (0,T] = \bigcup_{k=1}^K ((k-1)\tau,k\tau], \quad \tau = T/K.
$$
{}For given $(\rho_{k-1},c_{k-1})$, which approximates
$(\rho,c)$ at time $\tau(k-1)$, we wish to solve the approximate problem
in the weak formulation
\begin{align}
  \frac{1}{\tau} & \int_\Omega\big((\rho_k-\rho_{k-1})\phi
  + \alpha(c_k-c_{k-1})\psi\big)dx
  \nonumber \\
  &{}+ \int_\Omega\begin{pmatrix} \na\phi \\ \na\psi \end{pmatrix}^\top
  \begin{pmatrix}
  (m/n)\rho_k^{m-n+1} & -\delta\rho_k \\
  \delta\rho_k        & \delta
  \end{pmatrix}\begin{pmatrix} \na r_k \\ \na b_k \end{pmatrix}dx \label{2.semi} \\
  &{}+ \eps\int_\Omega(\Delta r_k\Delta\phi +|\nabla r_k|^2\nabla r_k\cdot\nabla \phi
  + r_k\phi)dx
  = \int_\Omega(\rho_k-c_k)\psi dx, \nonumber
\end{align}
where the entropy variables are given by
$$
  r_k = \frac{n}{n-1}\rho_k^{n-1}, \quad b_k = \frac{c_k}{\delta},
$$
and $(\phi,\psi)\in H^2(\Omega)\times H^1(\Omega)$ is a test function pair, well defined for $n>1$.
We prove now the existence of a solution to \eqref{2.semi} recalling that
\begin{equation}\label{2.p}
  p = \frac{m+n-1}{2}.
\end{equation}

\begin{remark}\rm\label{remark-regularisation}
For proving the existence of weak solutions, the regularization
$\eps\int_\Omega(\Delta r_k\Delta\phi +r_k\phi)dx$ would be
sufficient. The additional term is helpful when  deriving  energy
estimates which lead to the uniform boundedness of the solutions to the
parabolic-elliptic system, see Section \ref{sec.unif}.
\qed
\end{remark}

\begin{proposition}\label{prop.semi}
Let $\Omega\subset\R^d$ $(1\leq d\le 3)$ be a bounded domain with
$\pa\Omega\in C^{1,1}$. Furthermore, let
$(r_{k-1},$ $b_{k-1})\in
L^{n/(n-1)}(\Omega)\times L^2(\Omega)$ with $r_{k-1}\ge 0$ in $\Omega$ and
let $m>0$, $n>1$ be such that $1-2/d<p\le m$ with $p$ given by \eqref{2.p}.
Then there exists a weak solution
$(r_k,b_k)\in H^2(\Omega)\times H^1(\Omega)$ to \eqref{2.semi} satisfying
$r_k\ge 0$ in $\Omega$.
\end{proposition}

\begin{proof}
{\em Step 1: Formulation of a modified problem.}
In order to solve \eqref{2.semi} in terms of $(r,b)$, we set
$$
  w(r) = \rho = \left(\frac{n-1}{n}r\right)^{1/(n-1)}, \quad r_+ = \max\{0,r\}.
$$
We wish to solve first the system
\begin{align}
  \frac{1}{\tau} & \int_\Omega\big((w(r_+)-w(r_{k-1}))\phi
  + \alpha \delta(b_k-b_{k-1})\psi\big)dx
  \nonumber \\
  &{}+ \int_\Omega\begin{pmatrix} \na\phi \\ \na\psi \end{pmatrix}^\top
  \begin{pmatrix}
  (m/n)w(r_+)^{m-n+1} & -\delta w(r_+) \\
  \delta w(r_+)       & \delta
  \end{pmatrix}\begin{pmatrix} \na r \\ \na b \end{pmatrix}dx \label{2.semi.rb} \\
  &{}+ \eps\int_\Omega(\Delta r\Delta\phi +|\nabla r|^2\nabla r\cdot\nabla \phi+ r\phi)dx
  = \int_\Omega(w(r_+)-\delta b)\psi dx, \nonumber
\end{align}
where $(\phi,\psi)\in H^2(\Omega)\times H^1(\Omega)$.
Notice that the assumption $p\le m$ is equivalent to $m-n+1\ge 0$, which is needed for
the term $w(r_+)^{m-n+1}$ to be well defined. The minimum principle
shows that any weak solution $(r,b)$ to this problem satisfies $r\ge 0$ in $\Omega$.
Indeed, using $(r_-,0)$, where $r_-=\min\{0,r\}$, as a test function, and
observing that $w(r_+)\na r_-=0$, we obtain
$$
  -\frac{1}{\tau}\int_\Omega w(r_{k-1})r_- dx
  + \frac{m}{n}\int_\Omega w(r_+)^{m-n+1}|\na r_-|^2 dx
  + \eps\int_\Omega((\Delta r_-)^2 + |\nabla r_-|^4+r_-^2)dx = 0.
$$
Since all three integrals on the right-hand side are nonnegative, we conclude that
$r_-=0$ and $r\ge 0$ in $\Omega$.

{\em Step 2: The linearized problem.} Let $\sigma\in[0,1]$
and $(\bar r,\bar b)\in H^{7/4}(\Omega)
\times L^2(\Omega)$ be given. The Sobolev embedding $H^{7/4}(\Omega)\hookrightarrow
C^0(\overline\Omega)$ for $d\le 3$ shows that $w(\bar r_+)$ is
bounded. Hence, the following linear problem is well defined:
\begin{equation}\label{2.LM}
  a((r,b),(\phi,\psi)) = \sigma f(\phi,\psi) \quad\mbox{for all }
  (\phi,\psi)\in H^2(\Omega)\times H^1(\Omega),
\end{equation}
where
\begin{align*}
  a((r,b),(\phi,\psi))
  &= \int_\Omega\begin{pmatrix} \na\phi \\ \na\psi \end{pmatrix}^\top
  \begin{pmatrix}
  (m/n)w(\bar r_+)^{m-n+1} & -\delta w(\bar r_+) \\
  \delta w(\bar r_+)       & \delta
  \end{pmatrix}\begin{pmatrix} \na r \\ \na b \end{pmatrix}dx \\
  &\phantom{xx}{}+ \eps\int_\Omega(\Delta r\Delta\phi
  +|\nabla \bar r|^2\nabla r\cdot\nabla \phi+ r\phi)dx
  + \delta \int_\Omega b\psi dx, \\
  f(\phi,\psi) &= -\frac{1}{\tau}\int_\Omega\big((w(\bar r_+)-w(r_{k-1}))\phi
  + \alpha\delta(\bar b-b_{k-1})\psi\big)dx + \int_\Omega w(\bar r_+)\psi dx.
\end{align*}
The function $a:(H^2(\Omega)\times H^1(\Omega))^2\to\R$ is
bilinear and continuous due to the Sobolev embedding
$H^2(\Omega)\subset H^{7/4}(\Omega)\subset W^{1,4}(\Omega)\subset
C^0(\bar\Omega)$ for $d\leq 3$. Here, we need the assumption
$m-n+1\ge 0$. The function $f:H^2(\Omega)\times H^1(\Omega)\to\R$
is linear and bounded which is a consequence of the estimate
$$
  \int_\Omega w(r_{k-1})\phi dx
  \le \|w(r_{k-1})\|_{L^n(\Omega)}\|\phi\|_{L^{n/(n-1)}(\Omega)}
  \le C\|\phi\|_{H^2(\Omega)},
$$
for some constant $C>0$,
since $r_{k-1}\in L^{n/(n-1)}(\Omega)$ gives $w(r_{k-1})\in L^n(\Omega)$.
Moreover, $a$ is coercive:
\begin{align*}
  a((r,b),(r,b))
  &= \int_\Omega\Big(\frac{m}{n}w(\bar r_+)^{m-n+1}|\na r|^2
  + \delta|\na b|^2\Big)dx \\
  &\phantom{xx}{}+ \eps\int_\Omega\big((\Delta r)^2 +|\nabla \bar r|^2|\nabla r|^2+ r^2\big)dx
  + \delta\int_\Omega b^2 dx \\
  &\ge C(\eps,\delta)\big(\|r\|_{H^2(\Omega)}^2 + \|b\|_{H^1(\Omega)}^2\big),
\end{align*}
for some constant $C(\eps,\delta)>0$, since $\partial \Omega \in C^{1,1}$ (see Troianiello \cite{Tro}, p.~194).
The Lax-Milgram lemma now implies the existence and uniqueness of a solution
$(r,b)\in H^2(\Omega)\times H^1(\Omega)$ to \eqref{2.LM}.

{\em Step 3: The nonlinear problem.} The previous step allows us to define the
fixed-point operator $S:[0,1]\times H^{7/4}(\Omega)\times L^2(\Omega)\to
H^{7/4}(\Omega)\times L^2(\Omega)$ by $S(\sigma,\bar r,\bar b) = (r,b)$,
where $(r,b)\in H^2(\Omega)\times H^1(\Omega)$ is the unique solution to \eqref{2.LM}.
It holds $S(0,\bar r,\bar b)=(0,0)$ for all $(\bar r,\bar b)\in
H^{7/4}(\Omega)\times L^2(\Omega)$. Standard arguments prove that $S$ is
continuous and compact, taken into account the compact embedding of
$H^2(\Omega)\times H^1(\Omega)$ into $H^{7/4}(\Omega)\times L^2(\Omega)$.

It remains to show that there exists a constant $C>0$ such that for all fixed points
$(r,b)\in H^{7/4}(\Omega)\times L^2(\Omega)$ and $\sigma\in[0,1]$ satisfying
$S(\sigma,r,b)=(r,b)$, the estimate
\begin{equation}\label{2.unif}
  \|(r,b)\|_{H^{7/4}(\Omega)\times L^2(\Omega)} \le C
\end{equation}
holds. Let $(r,b)$ be such a fixed point. Let us first assume that
$\sigma=1$. Then $(r,b)$ is a solution to \eqref{2.semi}. By the
first step of the proof, we have $r\ge 0$ in $\Omega$. Moreover,
we can easily derive a uniform $L^1$ bound for $\rho=w(r)$ by
employing $(1,0)$ as a test function in \eqref{2.semi}:
$$
  \int_\Omega \rho dx = \int_\Omega \rho_{k-1}dx - \tau\eps\int_\Omega r dx
  \le \int_\Omega \rho_{k-1}dx,
$$
since $r$ is nonnegative. By iteration, we infer that
\begin{equation}\label{2.L1}
  \|\rho\|_{L^1(\Omega)} \le \|\rho_0\|_{L^1(\Omega)}.
\end{equation}
The uniform estimate \eqref{2.unif} is a consequence of the
following discrete entropy estimate, which settles the case $\sigma=1$.
The case $\sigma<1$ can be treated similarly.

\begin{lemma}\label{lem.ent}
Let $(r,b)\in H^2(\Omega)\times H^1(\Omega)$ be a
solution to \eqref{2.semi} and let $1-2/d<p\le m$. Then
\begin{align*}
  E(\rho,c) &+ \frac{\tau mn}{2p^2}\|\na \rho^p\|_{L^2(\Omega)}^2
  + \frac{\tau}{2\delta}\|c\|_{H^1(\Omega)}^2 \\
  &{}+ \eps\tau \big(\|\Delta r\|_{L^2(\Omega)}^2
  +\|\nabla r\|_{L^4(\Omega)}^4+\| r\|_{L^2(\Omega)}^2\big)
  \le E(\rho_{k-1},c_{k-1}),
\end{align*}
where the entropy $E(\rho,c)$ is defined in \eqref{1.ent},
$\rho=w(r)=((n-1)r/n)^{1/(n-1)}$, $c = \delta b$, and $p$ is defined in
\eqref{2.p}.
\end{lemma}

In order to prove this lemma, we employ the test function
$(r,b)=(n\rho^{n-1}/(n-1),c/\delta)$ in \eqref{2.semi}:
\begin{align}
  \frac{1}{\tau} & \int_\Omega\Big(\frac{n}{n-1}\rho^{n-1}(\rho-\rho_{k-1})
  + \frac{\alpha}{\delta}c(c-c_{k-1})\Big)dx \nonumber \\
  &{}+ \int_\Omega\Big(\frac{mn}{p^2}|\na\rho^p|^2
  + \frac{1}{\delta}(|\na c|^2+c^2)\Big) dx
  + \eps\int_\Omega((\Delta r)^2 +|\nabla r|^4+ r^2)dx = \frac{1}{\delta}\int_\Omega \rho c dx,
  \label{2.aux1}
\end{align}
where $\rho_{k-1}=w(r_{k-1})$ and $c_{k-1} = \delta b_{k-1}$.
Since $n>1$, the mapping $g(x)=x^n$, $x\ge 0$, is convex,
which implies the inequality $g(x)-g(y)\le g'(x)(x-y)$ for all $x$, $y\ge 0$. Hence,
the first integral on the left-hand side of \eqref{2.aux1} is bounded from below by
$$
  \frac{1}{\tau}\int_\Omega\Big(\frac{1}{n-1}(\rho^n-\rho_{k-1}^n)
  + \frac{\alpha}{2\delta}(c^2-c_{k-1}^2)\Big)dx = \frac{1}{\tau}(E(\rho,c)-E(\rho_{k-1},c_{k-1})).
$$
{}For the estimate of the right-hand side of \eqref{2.aux1}, we employ first
the H\"older and Young inequalities:
$$
  \frac{1}{\delta}\int_\Omega \rho cdx
  \le \frac{1}{\delta}\|\rho\|_{L^q(\Omega)}\|c\|_{L^{q'}(\Omega)}
  \le \frac{1}{2\delta}\|\rho\|_{L^q(\Omega)}^2
  + \frac{1}{2\delta}\|c\|_{H^1(\Omega)}^2,
$$
where $q\ge 6/5$ if $d=3$, $1<q<\infty$ if $d\le 2$, and $q'=q/(q-1)$.
In the last step, we have used the continuous embedding $H^1(\Omega)
\hookrightarrow L^{q'}(\Omega)$ which is valid since $q'\le 6$ if $d=3$
and $q'<\infty$ if $d\le 2$. By Lemma \ref{lem.ineq}, we find that
$$
  \frac{1}{\delta}\int_\Omega \rho cdx
  \le \frac{mn}{2p^2}\|\na\rho^p\|_{L^2(\Omega)}^2 + C(\delta,\|\rho\|_{L^1(\Omega)})
  + \frac{1}{2\delta}\|c\|_{H^1(\Omega)}^2.
$$
The assumptions of the lemma are clearly satisfied for $d\leq 2$.
If $d=3$ we can choose $q=2d/(d+2)=6/5$ for $p\leq 1$ and
$q=p+1/2>6/5$ for $p>1$. Putting together the above estimates and
the $L^1$ bound \eqref{2.L1}, this finishes the proof of Lemma
\ref{lem.ent} and of Proposition \ref{prop.semi}.
\end{proof}


\subsection{Uniform estimates}

Let $(r_k,b_k)$ be a solution to the approximated problem \eqref{2.semi} and set
$\rho_k=w(r_k)$, $c_k=\delta b_k$.
We define the piecewise constant functions
$$
  (\rho^{(\eps,\tau)},r^{(\eps,\tau)},c^{(\eps,\tau)})(x,t) = (\rho_k,r_k,c_k)(x)
  \quad\mbox{for }x\in\Omega,\ t\in((k-1)\tau,k\tau].
$$
We denote by $D_\tau\rho(\cdot,t) = (\rho(\cdot,t)-\rho(\cdot,t-\tau))/\tau$
the discrete time derivative of $\rho(\cdot,t)$, where $t\ge \tau$.
In terms of the variables $(\rho^{(\eps,\tau)},c^{(\eps,\tau)})$, system
\eqref{2.semi.rb} can be formulated as
\begin{align}
  0 &= \int_\tau^T \langle D_\tau\rho^{(\eps,\tau)},\phi\rangle dt
  + \int_\tau^T\int_\Omega\big(\na(\rho^{(\eps,\tau)})^m - \rho^{(\eps,\tau)}
  \na c^{(\eps,\tau)}\big)\cdot\na\phi dx\,dt \nonumber \\
  &\phantom{xx}{}+ \eps\int_\tau^T\int_\Omega(\Delta r^{(\eps,\tau)}\Delta\phi +|\nabla r^{(\eps,\tau)}|^2\nabla r^{(\eps,\tau)}\cdot\nabla \phi
  + r^{(\eps,\tau)}\phi) dx\,dt, \label{3.tau.rho} \\
  0 &= \alpha \int_\tau^T\langle D_\tau c^{(\eps,\tau)},\psi \rangle dt
  + \int_\tau^T\int_\Omega\big(\na c^{(\eps,\tau)} + \delta\na(\rho^{(\eps,\tau)})^n\big)
  \cdot\na\psi dx\,dt \nonumber \\
  &\phantom{xx}{}+ \int_\tau^T\int_\Omega(\rho^{(\eps,\tau)}-c^{(\eps,\tau)})\psi dx\,dt
  \label{3.tau.c}
\end{align}
for all smooth test functions $\phi$ and $\psi$, where $\langle\cdot,\cdot\rangle$
is a dual product.
We set $\Omega_T=\Omega\times(0,T)$ for given $T>0$.
Before we can perform the limit $(\eps,\tau)\to 0$, we need to prove some
uniform bounds in $\eps$ and $\tau$. The following result is a consequence of
the discrete entropy estimate of Lemma \ref{lem.ent} and the $L^1$ bound
\eqref{2.L1}, after integrating with respect to time.

\begin{lemma}\label{lem.est.ent}
Let $T>0$ and $1-2/d<p\le m$. Then the following uniform bounds hold:
\begin{align}
  \|\rho^{(\eps,\tau)}\|_{L^\infty(0,T;L^1(\Omega)\cap L^n(\Omega))}
  + \|\na(\rho^{(\eps,\tau)})^p\|_{L^2(\Omega_T)} &\le C, \label{2.est.rho} \\
  \sqrt{\eps}\|r^{(\eps,\tau)}\|_{L^2(0,T;H^2(\Omega))} +\sqrt[4]{\eps}\|\nabla r^{(\eps,\tau)}\|_{L^4(\Omega_T)}&\le C, \label{2.est.r} \\
  \alpha\|c^{(\eps,\tau)}\|_{L^\infty(0,T;L^2(\Omega))}
  + \|c^{(\eps,\tau)}\|_{L^2(0,T;H^1(\Omega))} &\le C, \label{2.est.c}
\end{align}
where $C>0$ is here and in the following a generic constant independent of
$\eps$ and $\tau$, and $p$ is defined in \eqref{2.p}.
\end{lemma}

Under additional assumptions on the exponents $n$ and $p$,
we are able to derive more a priori estimates.

\begin{lemma}\label{lem.est.s}
Let $p\le\min\{m,n\}$ and $Q:=n/d+p>1$ and set
$$
  s_1 = \frac{2Q}{Q+m-p} \in(1,2], \quad s_2 = \frac{2Q}{Q+n-p} \in(1,2], \quad
  s_3 = \frac{2Q}{Q+1} \in(1,2).
$$
Then the following uniform bounds hold:
\begin{align}
  \|(\rho^{(\eps,\tau)})^p\|_{L^2(0,T;H^1(\Omega))}
  + \|\rho^{(\eps,\tau)}\|_{L^{2Q}(\Omega_T)} &\le C, \label{2.est.p} \\
  \|(\rho^{(\eps,\tau)})^m\|_{L^{s_1}(0,T;W^{1,s_1}(\Omega))}
  + \|(\rho^{(\eps,\tau)})^n\|_{L^{s_2}(0,T;W^{1,s_2}(\Omega))} &\le C,
  \label{2.est.mn} \\
  \|\rho^{(\eps,\tau)}\na c^{(\eps,\tau)}\|_{L^{s_3}(\Omega_T)} &\le C,
  \label{2.est.rhoc} \\
  \|D_\tau\rho^{(\eps,\tau)}\|_{L^{\tilde s}(\tau,T;(H^{3}(\Omega))')}
  + \alpha \|D_\tau c^{(\eps,\tau)}\|_{L^{s}(\tau,T;(H^{3}(\Omega))')} &\le C,
  \label{2.est.tau}
\end{align}
where $s=\min\{s_1,s_2,s_3\}$, $\tilde s=\min\{s,4/3\}$, and $p$ is defined in \eqref{2.p}.
\end{lemma}

We remark that the condition $Q>1$ is equivalent to $p>1-n/d$,
which in particular implies the condition $p>1-2/d$ as explained
in the introduction after \eqref{1.aux}.

\begin{proof}
We set $\rho=\rho^{(\eps,\tau)}$ and $c=c^{(\eps,\tau)}$ to simplify the notation.
By the Poincar\'e inequality, we find that
$$
  \|\rho^p\|_{L^2(\Omega_T)}^2
  = \int_0^T\|\rho^p\|_{L^2(\Omega)}^2 dt
  \le C\int_0^T\|\na\rho^p\|_{L^2(\Omega)}^2 dt
  + C\int_0^T\|\rho^p\|_{L^1(\Omega)}^2dt.
$$
Since $\rho$ is uniformly bounded in $L^\infty(0,T;L^n(\Omega))$ and
since we have assumed that $p\le n$, the right-hand side of the above inequality
is uniformly bounded. This shows the uniform bound for $\rho^p$ in
$L^2(0,T;H^1(\Omega))$.
Next, the Gagliardo-Nirenberg inequality with $\theta=p/Q<1$ gives
\begin{align*}
  \|\rho\|_{L^{2Q}(\Omega_T)}^{2Q}
  &= \|\rho^p\|_{L^{2Q/p}(\Omega_T)}^{2Q/p}
  \le C\int_0^T\|\rho^p\|_{H^1(\Omega)}^{2Q\theta/p}
  \|\rho^p\|_{L^{n/p}(\Omega)}^{2Q(1-\theta)/p}dt \\
  &\le C\|\rho\|_{L^\infty(0,T;L^{n}(\Omega))}^{2Q(1-\theta)}
  \int_0^T\|\rho^p\|_{H^1(\Omega)}^2 dt \le C.
\end{align*}
This shows \eqref{2.est.p}.

For the proof of \eqref{2.est.mn}, we observe that $s_1>1$ is equivalent
to $n/d+n>1$, which is true since $n>1$, and that $s_1\le 2$ is equivalent
to $p\le m$, which holds by assumption.
Hence $s_1\in(1,2]$. Let first $s_1<2$. We apply the H\"older inequality with
exponents $\gamma=2/s_1>1$ and $\gamma'=2/(2-s_1)$:
\begin{align*}
  \|\na(\rho^m)\|_{L^{s_1}(\Omega_T)}^{s_1}
  &= \Big(\frac{m}{p}\Big)^{s_1}\int_0^T\int_\Omega \rho^{(m-p)s_1}
  |\na\rho^p|^{s_1} dx\,dt \\
  &\le C\|\rho\|_{L^{2Q}(\Omega_T)}^{(m-p)s_1}\|\na\rho^p\|_{L^2(\Omega_T)}^{s_1}
  \le C,
\end{align*}
because of \eqref{2.est.p}. If $s_1=2$, it follows that $m=p$, and the
conclusion still holds. The estimate for $\rho^m$ is shown in a similar way
by applying the H\"older inequality with exponents $\gamma=2/s_1>1$ (if $s_2<2$) and
$\gamma'=2/(2-s_1)$ to $\rho^m = \rho^{m-p}\rho^p$. Hence, $\rho^m$ is uniformly
bounded in $L^{s_1}(0,T;W^{1,s_1}(\Omega))$.

The estimates for $\na(\rho^n)$ and $\rho^n$ in $L^{s_2}(\Omega_T)$ are proved
analogously by replacing $s_1$ by $s_2$. The relation $s_2\le 2$ is equivalent
to $p\le n$, and $s_2>1$ is equivalent to $n/d+m>1$, which is true since
$n/d+m\ge n/d+p>1$ by assumption. This proves \eqref{2.est.mn}.

Estimate \eqref{2.est.rhoc} is again a consequence of the H\"older
inequality, with exponents $\alpha=2/s_3$ and $\alpha'=2/(2-s_3)$:
$$
  \|\rho\na c\|_{L^{s_3}(\Omega_T)}
  \le \|\rho\|_{L^{2s_3/(2-s_3)}(\Omega_T)}\|\na c\|_{L^2(\Omega_T)}
  = \|\rho\|_{L^{2Q}(\Omega_T)}\|\na c\|_{L^2(\Omega_T)} \le C.
$$

We turn now to the estimates for the discretized time derivatives.
Let $\phi\in L^{\tilde s'}(0,T;$ $H^3(\Omega))$, where $s'=s/(s-1)\ge 2$ and $\tilde s'=\max\{s',4\}$. Then,
using $1<s\le 2$ and \eqref{2.est.r}, \eqref{2.est.mn}, and \eqref{2.est.rhoc},
\begin{eqnarray*}
 && \left|\int_\tau^T\langle D_\tau\rho,\phi\rangle dt\right|\\
  &&\quad= \left|-\int_\tau^T\int_\Omega\big(\na(\rho^m)-\rho\na c\big)\cdot\na\phi dxdt
  -\eps\int_\tau^T\int_\Omega\big(\Delta r\Delta\phi +|\nabla r|^2\nabla r\cdot\nabla\phi+ r\phi)dxdt
  \right| \\
  &&\quad\le \big(\|\na(\rho^m)\|_{L^s(\Omega_T)}
  + \|\rho\na c\|_{L^s(\Omega_T)}\big)
  \|\na\phi\|_{L^{s'}(\Omega_T)} \\
  &&\quad\phantom{xx}{}+ \eps\|\Delta r\|_{L^2(\Omega_T)}\|\Delta\phi\|_{L^2(\Omega_T)} +\eps \|\nabla r\|^3_{L^4(\Omega_T)}\|\nabla \phi\|_{L^4(\Omega_T)}
  + \eps\|r\|_{L^2(\Omega_T)}\|\phi\|_{L^2(\Omega_T)} \\
  &&\quad\le C\|\phi\|_{L^{\tilde s'}(0,T;H^3(\Omega))}.
\end{eqnarray*}
Furthermore, since $\r\in L^{2Q}(\Omega_T)\subset L^2(\Omega_T)$, we obtain
\begin{align*}
  \left|\int_\tau^T\langle D_\tau c,\phi\rangle dt\right|
  &= \left|-\int_\tau^T\int_\Omega\big(\na c+\delta\na(\rho^n)\big)\cdot\na\phi dxdt
  + \int_\tau^T\int_\Omega(\rho-c)\phi dxdt\right| \\
  &\le \big(\|\na c\|_{L^s(\Omega_T)}+ \delta\|\na(\rho^n)\|_{L^s(\Omega_T)}\big)
  \|\nabla\phi\|_{L^{s'}(\Omega_T)} \\
  &\phantom{xx}{}+ \big(\|\rho\|_{L^2(\Omega_T)} + \|c\|_{L^2(\Omega_T)}\big)
  \|\phi\|_{L^2(\Omega_T)} \\
  &\le C\|\phi\|_{L^{s'}(0,T;H^3(\Omega))}.
\end{align*}
using \eqref{2.est.c}, \eqref{2.est.p}, and \eqref{2.est.mn}.
Thus we have proved \eqref{2.est.tau}.
\end{proof}


\subsection{The limit of vanishing approximation parameters.}

We show first the strong convergence of the sequence $(\rho^{(\eps,\tau)})$.

\begin{lemma}
Let the assumptions of Lemma \ref{lem.est.s} hold. Then, up to a subsequence,
\begin{equation}\label{2.conv.2Q}
  \rho^{(\eps,\tau)} \to \rho \quad\mbox{strongly in }L^{2r}(0,T;L^{2Q}(\Omega)),
\end{equation}
for all $r<Q$, where $Q=n/d+p>1$.
\end{lemma}

\begin{proof}
First, consider $p < 1$. We show that $(\rho^{(\eps,\tau)})$
is bounded in $L^\ell(0,T;$ $W^{1,\ell}(\Omega))$, where $\ell=2Q/(Q+1-p)>1$.
Notice that $1<\ell<2$ since the former inequality is equivalent to $Q>1-p$ which
is true since $Q>1$ by assumption, and the latter property is equivalent to $p<1$.
We apply the H\"older inequality to
$\na\rho^{(\eps,\tau)} = (1/p)(\rho^{(\eps,\tau)})^{1-p}\na(\rho^{(\eps,\tau)})^p$
yielding
\begin{align*}
  \|\na\rho^{(\eps,\tau)}\|_{L^\ell(\Omega_T)}
  &\le C\|(\rho^{(\eps,\tau)})^{1-p}\|_{L^{2\ell/(2-\ell)}(\Omega_T)}
  \|\na(\rho^{(\eps,\tau)})^p\|_{L^2(\Omega_T)} \\
  &= C\|\rho^{(\eps,\tau)}\|_{L^{2Q}(\Omega_T)}^{1-p}
  \|\na(\rho^{(\eps,\tau)})^p\|_{L^2(\Omega_T)},
\end{align*}
since $2\ell(1-p)/(2-\ell)=2Q$. Thus, taking into account the bound
for $(\rho^{(\eps,\tau)})$ in $L^{2Q}(\Omega_T)$, by \eqref{2.est.p},
this proves the desired bound.

Next, we claim that the embedding $W^{1,\ell}(\Omega)\hookrightarrow L^{2Q}(\Omega)$
is compact. This is the case if $1-d/\ell>-d/(2Q)$, which is equivalent to
$2n/d+m>1$. In order to show this inequality, we observe that the assumption
$p>1-2/d$ is equivalent to $n+m>3-4/d$. If $d=2$, this implies that
$2n/d+m=n+m>1$; if $d=3$, we find that
$$
  \frac{2n}{d}+m = \frac23(n+m) + \frac{m}{3} > \frac23(n+m) > \frac23\cdot\frac53 > 1.
$$
Hence, $1-d/\ell>-d/(2Q)$ is satisfied for $d\le 3$.
In view of the bound for the discrete time
derivative of $(\rho^{(\eps,\tau)})$, see \eqref{2.est.tau}, Lemma \ref{lem.comp}
(take $p=1$ in the lemma) implies the existence of a subsequence of
$(\rho^{(\eps,\tau)})$ (not relabeled) such that \eqref{2.conv.2Q} holds.

Finally, if $p\ge 1$, we can apply Lemma \ref{lem.comp} to conclude the strong
convergence result.
\end{proof}

The estimates of Lemmas \ref{lem.est.ent} and \ref{lem.est.s} imply
the existence of a subsequence of $(\rho^{(\eps,\tau)})$, which is not relabeled,
such that, as $(\eps,\tau)\to 0$,
\begin{align*}
  c^{(\eps,\tau)} \rightharpoonup c &\quad\mbox{weakly in }L^2(0,T;H^1(\Omega)), \\
  (\rho^{(\eps,\tau)})^m \rightharpoonup z_1 &\quad\mbox{weakly in }
  L^{s_1}(0,T;W^{1,s_1}(\Omega)), \\
  (\rho^{(\eps,\tau)})^n \rightharpoonup z_2 &\quad\mbox{weakly in }
  L^{s_2}(0,T;W^{1,s_2}(\Omega)), \\
  \rho^{(\eps,\tau)}\na c^{(\eps,\tau)} \rightharpoonup z_3 &\quad\mbox{weakly in }
  L^{s_3}(0,T;W^{1,s_3}(\Omega)), \\
  D_\tau\rho^{(\eps,\tau)} \rightharpoonup \pa_t\rho &\quad\mbox{weakly in }
  L^{\tilde s}(0,T;(H^{3}(\Omega))'), \\
  D_\tau c^{(\eps,\tau)} \rightharpoonup \pa_t c &\quad\mbox{weakly in }
  L^{s}(0,T;(H^{3}(\Omega))') \ \ \textnormal {if} \ \alpha>0, \\
  \eps r^{(\eps,\tau)} \rightharpoonup 0 &\quad\mbox{weakly in }L^2(0,T;H^2(\Omega)),\\
  \eps |\nabla r^{(\eps,\tau)}|^2\nabla r^{(\eps,\tau)} \rightharpoonup 0 &\quad\mbox{weakly in }L^{4/3}(\Omega_T).
\end{align*}
The limits of the nonlinearities are easily identified since (a subsequence of)
$(\rho^{(\eps,\tau)})$ converges strongly in $L^{2Q}(\Omega_T)$, where $Q > 1$.
Hence, up to a subsequence, $\rho^{(\eps,\tau)}\to \rho$ a.e.\ in $\Omega_T$
and $(\rho^{(\eps,\tau)})^m\to \rho^m$, $(\rho^{(\eps,\tau)})\to \rho^n$ a.e.,
implying that $z_1=\rho^m$, $z_2=\rho^n$. Moreover, the strong convergence of
$(\rho^{(\eps,\tau)})$ in $L^2(\Omega_T)$ and the weak convergence of
$\na c^{(\eps,\tau)}$ in $L^2(\Omega_T)$ give the weak convergence of
$(\rho^{(\eps,\tau)}\na c^{(\eps,\tau)})$ to $\rho\na c$ in $L^1(\Omega_T)$,
implying that $z_3=\rho\na c$.

The above convergence results are sufficient to pass to the limit
$(\eps,\tau)\to 0$ in \eqref{3.tau.rho}-\eqref{3.tau.c}
leading to \eqref{1.rho}-\eqref{1.c}.
The Neumann boundary conditions are satisfied in the weak sense, and the
initial conditions hold in the sense of $L^{\tilde s}(0,T;(H^{3}(\Omega))')$.
Since in the limiting equation the regularizing terms vanish, test functions in
$L^{s'}(0,T;W^{1,s'}(\Omega))$ are sufficient to obtain the boundedness of the diffusion
and drift terms. A density argument now completes the proof.


\section{The parabolic-elliptic system}\label{sec.unif}
The parabolic-elliptic system corresponding to (\ref{1.c}) is
given by
\begin{equation}\label{par-ell}
  \pa_t \rho = \Delta \rho^m -\diver(\rho \nabla c), \quad
  0 = \Delta c+\delta \Delta \rho^n +\rho-c,
\end{equation}
subject to the no-flux boundary conditions
\begin{equation}\label{bc-parell}
  (\nabla \rho^m - \r\nabla c)\cdot \nu =0, \quad
  \nabla (c +\delta \rho^n)\cdot \nu= 0 \quad \textnormal{on } \partial \Omega,\ t>0,
\end{equation}
and the initial condition
\begin{equation}\label{id-parell}
  \rho(\cdot,0)=\rho_0 \quad \textnormal{in } \Omega.
\end{equation}
Similarly as in \cite{HiJu11} we introduce a new unknown corresponding
to the diffusion terms in the second equation,
$v=c+\delta\rho^n$,
and rewrite system (\ref{par-ell}) in terms of $\rho$ and $v$:
\begin{align}\label{par-ell-v}
  \begin{array}{lcl}
  \pa_t \rho&=&\Delta \Big(\rho^m+\delta\dfrac{n}{n+1}\rho^{n+1}\Big)
  -\diver (\rho \nabla v),\medskip\\
  0&=&\Delta v+\rho+\delta \rho^n-v,
  \end{array}
\end{align}
subject to the no-flux boundary conditions and
initial data \eqref{id-parell}. In the case $n=1$ one
simply obtains the Keller-Segel model with nonlinear diffusion,
which is known to prevent blow up \cite{CaCa06,Kow05,KoSz08}. Here,
the situation is different, since for $n>1$ we obtain additionally a
nonlinear growth term in the equation for $v$. We will show
that nevertheless this system satisfies the properties used by
Kowalczyk in \cite{Kow05} to obtain an $L^\infty(\Omega_T)$ bound for
$\rho$. The difference here is that instead of the conservation of
mass we have to make use of the uniform boundedness of
$\|\rho\|_{L^n(\Omega)}$, which only holds for finite times.

The main advantage of the parabolic-elliptic system is that it
allows for another entropy, since we will show that we can use
powers of $\rho$ as test functions in the elliptic equation, see Remark \ref{rem.weak}.

\begin{lemma}
Let $(\r,v)$ be the solution of the parabolic-elliptic system
\eqref{par-ell} constructed in Theorem \ref{thm.ex} with $\r_0 \in
L^n(\Omega)$ and $\r_0\geq 0$. Then
\[
  \r \geq 0, \quad v\geq 0 \quad  \textnormal{a.e. in } \Omega_T,
\]
and  $\r$ satisfies the additional entropy estimate
\begin{equation}\label{par-ell-energy}
  \sup_{t\in(0,T)}\int_\Omega \frac{\r(\cdot,t)^n}{n-1} dx
  +\frac{nm}{p^2}\|\nabla \rho^p\|_{L^2(\Omega_T)}^2 + \frac{\delta}{2}\|\nabla \rho^n\|^2_{L^2(\Omega_T)}\leq C \,.
\end{equation}
\end{lemma}

\begin{proof}
We start again from the regularized problem to derive the
nonnegativity of $v$ and the additional entropy estimate
rigorously. As in Section \ref{sec.ex} we skip the index $k$:
\begin{align}
  \int_\Omega D_\tau \r \phi  dx
  &= -\int_\Omega \nabla \Big(\r^m +\delta \frac{n}{n+1}\r^{n+1}\Big)
  \cdot \nabla \phi dx
  + \int_\Omega \r\nabla v\cdot \nabla \phi dx   \nonumber\\
  \label{par-ell-reg}
  &\phantom{xx}{}
  -\eps\int_\Omega \big(\Delta r \Delta \phi +|\nabla r|^2\nabla r\cdot \nabla \phi
  + r\phi \big) dx, \\
  0 &= -\int_\Omega \nabla v \cdot \nabla \psi dx
  + \int_\Omega (\r+\delta \r^n)\psi dx - \int_\Omega v \psi dx,   \nonumber
\end{align}
for appropriate test functions $\phi$ and $\psi$.
The existence of a global weak solution is proven in Proposition
\ref{prop.semi}. The a priori nonnegativity of $v$ follows from a
standard argument by testing the equation for $v$ with
$v^-=\min\{0,v\}$, see, e.g., \cite{Kow05}. The
nonnegativity is clearly preserved when performing the limit of
vanishing parameters.

To derive the additional energy estimate we first use
$r=\frac{n}{n-1}\r^{n-1}$ as a test function in the equation for
$\r$, integrate in time, and insert the elliptic equation for $v$:
\begin{align*}
  \int_\Omega & \frac{\r^n}{n-1} dx -\int_\Omega \frac{\r_0^n}{n-1}dx
  +\frac{nm}{p^2}\|\nabla \rho^p\|_{L^2(\Omega_T)}^2
  +  \delta\|\nabla \rho^n\|^2_{L^2(\Omega_T)}\\
  &\leq \int_{\Omega_T}\nabla v\cdot\nabla \rho^n dx\, dt
  -\eps \big(\|\Delta r\|_{L^2(\Omega_T)}^2
  +\|\nabla r\|_{L^4(\Omega_T)}^4+\|r\|_{L^2(\Omega_T)}^2\big) \\
  &\leq \int_{\Omega_T}\nabla v\cdot\nabla \rho^n dx\, dt
  = \int_{\Omega_T} \rho^{n+1} dx\, dt
  +\delta \int_{\Omega_T} \rho^{2n} dx\, dt
  - \int_{\Omega_T} v\rho^n dx\, dt \\
  &\leq C +  \frac{\delta}{2}\|\nabla\rho^n\|^2_{L^2(\Omega_T)}.
\end{align*}
Here, we have used the nonnegativity of $\r$ and $v$ and the
uniform boundedness of $\|\r\|_{L^n(\Omega)}$ (with respect to $k$)
together with the Gagliardo-Nirenberg inequalities in the following way:
\begin{align*}
  \int_{\Omega} \r^{n+1}  dx
  &=\|\r^n\|^{(n+1)/n}_{L^{(n+1)/n}(\Omega)}
  \leq C\|\nabla\r^n\|^{(n+1)\theta_1/n}_{L^2(\Omega)}\|\r^n\|^{(n+1)
  (1-\theta_1)/n}_{L^1(\Omega)}+\|\r^n\|^{(n+1)/n}_{L^1(\Omega)}\\
  &\leq  C\|\nabla\r^n\|^{(n+1)\theta_1/n}_{L^2(\Omega)}
  +C\leq \frac{\delta}{4}\|\nabla \rho^n\|^2_{L^2(\Omega)} + C(\delta),
\end{align*}
where $\theta_1=2d/((n+1)(2+d))\in(0,1)$, and for $\theta_2=d/(d+2)$,
we obtain the bound
$$
  \int_{\Omega} \r^{2n}  dx
  =\|\r^n\|^{2}_{L^{2}(\Omega)}
  \leq C\|\nabla\r^n\|^{2\theta_2}_{L^2(\Omega)}\|\r^n\|^{2(1-\theta_2)}_{L^1(\Omega)}
  +\|\r^n\|^{2}_{L^1(\Omega)}
  \leq\frac{\delta}{4}\|\nabla \rho^n\|^2_{L^2(\Omega)} + C(\delta).
$$
Since the constant $C(\delta)$ is independent of $(\eps, \tau)$, we can
perform the limit of vanishing parameters, which completes the
proof. \end{proof}

\begin{theorem}\label{par-ell-thm}
Let $\r_0\in L^\infty(\Omega)$, where $\Omega \subset\R^d$ is a bounded
domain $(1\leq d\leq 3)$ with $\partial \Omega\in
C^{1,1}$, and let the assumptions \eqref{conditions} hold. Then
the parabolic-elliptic system \eqref{par-ell}-\eqref{id-parell}
has a global weak solution satisyfing
\[
  \|\r\|_{L^\infty(0,T;L^\infty(\Omega))}
  + \|c\|_{L^\infty(0,T;L^\infty(\Omega))}\leq C(T)
\]
for any $T>0$.
\end{theorem}

In particular, test functions in $L^2(0,T;H^1(\Omega))$
are admissible for \eqref{par-ell}.

\begin{proof}
The iterative method of Alikakos used by
Kowalczyk to derive an $L^\infty(\Omega_T)$ bound for $\r$
requires test functions of the form $\r^q$ for $q\geq n$ as well
as the uniform boundedness of $\nabla v$. Thus, we prove first that
$\na v\in L^\infty(\Omega_T)$ and $\rho\in L^\infty(0,T;L^q(\Omega))$
for suitable $q>1$.

\emph{Step 1: Proof of $v\in L^\infty(0,T;W^{1,\infty}(\Omega))$.} If  $\r\in L^\infty(0,T;L^{3n+1}(\Omega))$, then elliptic regularity for
$$
  -\Delta v+v=\r+\delta \r^n\quad \textnormal{in }\Omega, \quad
  \nabla v\cdot \nu =0 \quad \textnormal{in }\partial\Omega
$$
implies that $v\in L^\infty(0,T;W^{2,3+1/n}(\Omega))$ (see, e.g., \cite{Gri}, p.~126), and hence, by  Sobolev embedding, $v\in L^\infty(0,T;W^{1,\infty}(\Omega))$.
In order to show that  $\r\in L^\infty(0,T;L^{3n+1}(\Omega))$, we employ in the regularized equation (\ref{par-ell-reg})  the test function $(3n+1) \r^{3n}$ and
integrate in time:
\begin{align*}
  \int_\Omega & \r^{3n+1} dx -\int_\Omega \r_0^{3n+1}dx
  + 2a\|\nabla \r^{(m+3n)/2}\|_{L^{2}(\Omega_T)}^2
  +2 \delta b\|\nabla\r^{2n+1/2}\|^2_{L^2(\Omega_T)}\\
  &\phantom{xx}{}+ (3n+1)\eps\left(\int_{\Omega_T}
  \Delta r\Delta \r^{3n}+|\nabla r|^2\nabla r\cdot\nabla \r^{3n}+r\r^{3n}\right)dx\,dt \\
  &\le 3n\int_{\Omega_T} \nabla \r^{3n+1}\cdot\nabla v dx dt\\
  &= 3n\int_{\Omega_T}\r^{3n+2}dx dt + 3\delta n\int_{\Omega_T} \r^{4n+1} dx\,dt
  - 3n \int_{\Omega_T} \r^{3n+1}v dx\,dt\\
  &\leq 4\delta n \int_{\Omega_T} \r^{4n+1} dx\,dt + C,
\end{align*}
where we have used the nonnegativity of $\r$ and $v$, H\"older's inequality,
and $3n+2<4n+1$. Moreover, we have set $a=6mn(3n+1)/(m+3n)^2$ and $b=6n^2(3n+1)/(4n+1)^2$. We apply the Gagliardo-Nirenberg inequality
with $\theta =d/(d+2)$ and H\"older's and Young's inequalities  to estimate
\begin{align*}
  \int_\Omega \r^{4n+1} dx
  &= \|\r^{2n+1/2}\|_{L^2(\Omega)}^2
  \leq C\|\nabla\r^{2n+1/2}\|_{L^2(\Omega)}^{2\theta}
  \|\rho^{2n+1/2}\|_{L^1(\Omega)}^{2(1-\theta)}
  + \|\r^{2n+1/2}\|_{L^1(\Omega)}^2 \\
  &\leq \frac{b}{4n}\|\nabla \r^{2n+1/2}\|_{L^2(\Omega)}^{2}
  + C \|\r^{2n+1/2}\|_{L^1(\Omega)}^2.
\end{align*}
The last summand is estimated by interpolating, for any $\beta>0$,
$$
  \|\r^{2n+1/2}\|_{L^1(\Omega)}^2
  \leq\|\r^{(3n+1)/2}\|_{L^2(\Omega)}^2\|\r^{n}\|_{L^1(\Omega)}
  \leq C\int_\Omega\r^{3n+1}dx\leq \beta \int_\Omega \r^{4n+1}dx +C(\beta),
$$
since $\r\in L^{\infty}(0,T;L^n(\Omega))$.
Finally, combining the above estimates and by choosing $\beta$ appropriately, we obtain
\[
  \int_{\Omega_T} \r^{4n+1} dx dt
  \leq \frac{b}{4n}\|\nabla \r^{2n+1/2}\|_{L^2(\Omega_T)}^{2}+ C.
\]

It remains to bound from below the terms arising from the regularization:
\begin{align*}
  \eps\int_{\Omega_T} & \big( \Delta r\Delta \r^{3n}+|\nabla r|^2\nabla r\cdot
  \nabla \r^{3n}+r\r^{3n}\big)dxdt \\
  &\geq 3 \eps\int_{\Omega_T} \left(\r^{2n+1}(\Delta r)^2 + 2\Big(1+\frac{1}{2n}\Big)
  \r^{n+2}|\nabla r|^2\Delta r +\r^{2n+1} |\nabla r|^4\right) dx\,dt \\
  &= 3 \eps\int_{\Omega_T} \left(\r^{n+1/2}\Delta r
  +\Big(1+\frac{1}{2n}\Big)\r^{3/2}|\nabla r|^2\right)^2dx\,dt\\
  &\phantom{xx}{} + 3\eps\int_{\Omega_T}\r^{3}\left(\r^{2(n-1)}
  - \Big(1+\frac{1}{2n}\Big)^2\right)|\nabla r|^4dx\,dt
\end{align*}\begin{align*}
  &\geq \eps\int_{\{\r^{n-1}\leq 1+1/2n\}}
  \r^3\left(\r^{2(n-1)}-\Big(1+\frac{1}{2n}\Big)^2\right) |\nabla r|^4 dx\,dt\\
  &\geq - \eps C\|\nabla r\|_{L^4(\Omega_T)}^4\geq -C\, ,
\end{align*}
where $C$ is independent of $(\eps, \tau)$ due to the basic entropy estimate for the regularized system.
Thus, we obtain in the limit $(\eps,\tau)\rightarrow 0$ for the weak  solution of (\ref{par-ell}),
\[
  \int_\Omega \r^{3n+1} dx  +\delta b\|\nabla\r^{2n+1/2}\|^2_{L^2(\Omega_T)}\leq C.
\]
This implies that $\r\in L^\infty(0,T;L^{3n+1}(\Omega))$, which proves the claim.

\emph{Step 2:} {\em Test functions in $L^2(0,T;H^1(\Omega))$ are admissible.}
We have to verify that
$\nabla \r^m\in L^2(0,T;L^2(\Omega))$. We recall that
the restrictions on the exponents are $n-1\leq m\leq n+1$ or, equivalently,
$p\leq m\leq n+1\leq 2n+1/2$. Therefore, in view of Step 1, we can interpolate
\begin{align*}
\frac{p^2}{m^2}\int_{\Omega_T}|\nabla \r^m|^2dx dt
  &=\int_{\Omega_T} \r^{2(m-p)}|\nabla \r^p|^2 dx dt \\
  &\leq \int_{\{\r\leq 1\}}|\nabla
  \r^{p}|^2dx\,dt+\int_{\{\r\geq1\}}\r^{2(2n+1/2-p)}|\nabla \r^p|^2dx\,dt\\
  &\leq \|\nabla \r^p\|_{L^2(\Omega_T)}^2
  +C\|\nabla\r^{2n+1/2}\|^2_{L^2(\Omega_T)}\leq C.
\end{align*}

\emph{Step 3:} We now proceed to make the estimates derived by Kowalczyk rigorous. To this end, we will use powers of the cut-off functions
$\r_K=\min\{\r,K\}$ as test functions. Due to (\ref{par-ell-energy}), $\r_K^q \in
L^2(0,T;H^1(\Omega))$ for any $q\geq n$ and according to Step 2, it is an
admissible test function. Let us introduce the notation
\[
  w(x,t)=(\r_K-k)_+ \quad \textnormal{for some } k>0.
\]
We test the equation for the cell density with $(q+1)w^q $, where $q\geq n$:
\begin{align*}
  \pa_t\int_{\Omega} w^{q+1} dx
  &=-\delta (q+1)\int_{\Omega} \nabla \r^n\cdot\nabla w^qdx
  - (q+1)\int_{\Omega} \nabla \r^{m} \cdot\nabla w^q dx \\
  &\phantom{xx}{} + (q+1)\int_{\Omega}  \r \nabla w^q\cdot\nabla v dx\\
  &\leq -\delta n(q+1)k^{n-1}\int_{\Omega} \nabla \r \cdot \nabla w^q dx
  -m q(q+1)\int_{\Omega} \ \r^{m-1} w^{q-1} |\nabla \r|^2dx \\
  &\phantom{xx}{}+ 2q \|\nabla v\|_{L^\infty(\Omega)}\left(\int_\Omega w^{(q+1)/2}
  |\nabla w^{(q+1)/2}| dx + k\int_{\Omega} w^{(q-1)/2} |\nabla w^{(q+1)/2}| dx
  \right)
\end{align*}
Neglecting the second term on the right-hand side and employing Young's inequality
to the last two terms, we arrive at
$$
  \pa_t\int_{\Omega} w^{q+1} dx
  \leq -\delta \frac{2nq}{q+1}k^{n-1}\int_{\Omega}
  |\nabla w^{(q+1)/2}|^2dx + C q(q+1)\int_{\Omega} w^{q+1} dx + Cq(q+1),
$$
where $C>0$ depends on, e.g., $\delta$ and $\|\nabla v\|_{L^\infty(\Omega)}$.
Starting from this  inequality, Kowalczyk \cite{Kow05} employed the iterative method of Alikakos to obtain
\[
  \|w(\cdot,t)\|_{L^\infty(\Omega)}\leq C(\|w_0\|_{L^\infty(\Omega)})
  \leq C(\|\r_0\|_{L^\infty(\Omega)}) \quad\textnormal{for all }0<t<T.
\]
Hence, for $\r_0\in L^{\infty}(\Omega_T)$, we have
$w=(\r_K-k)_+\in L^{\infty}(\Omega_T)$ for any $K$, $k>0$ with an independent bound for the norm. We can let $K\rightarrow \infty$ to deduce
\[
  \|\r(\cdot,t)\|_{L^\infty(\Omega)}\leq C(\|\r_0\|_{L^\infty(\Omega)})
  \quad \textnormal{for all }0<t<T,
\]
which finishes the proof.
\end{proof}


\section{Long-time behaviour}\label{sec.lt}

The (modified) Keller-Segel system possesses the constant homogeneous steady state
$\r^*=c^*=M/\mbox{meas}(\Omega)$.
Let us consider the following system, which is equivalent to \eqref{1.rho}-\eqref{1.c}:
\begin{align}
  \label{decay.rho}\pa_t\r &= \Delta \r^m -\diver(\r\nabla c), \\
  \label{decay.c}\alpha \pa_t c &= \Delta c+\delta \Delta \r^n +\r - \r^* - (c -c^*)
  \quad \mbox{in }\Omega,\ t>0.
\end{align}
In the case of linear diffusion terms, $m=n=1$, the decay of the relative entropy yields the convergence of the solution towards the homogeneous steady state for large enough $\delta$ if $d=2$ \cite{HiJu11}. The corresponding relative entropy for the system under consideration with nonlinear diffusion is
\[
  E^*(t)=\int_\Omega \left[\frac{(\r-\r^*)^n}{n-1}+\alpha \frac{(c-c^*)^2}{2\delta}
  \right] dx.
\]
Notice that the nonnegativity of $E^*$ is only guaranteed if $n$ is an even integer. In particular, $E^*$ is not well defined for general real $n>1$, since $\r-\r^*$ may be
negative. Formally, testing \eqref{decay.rho}-\eqref{decay.c} with $(n(\r-\r^*)^{n-1}/(n-1), (c-c^*)/\delta)$ we obtain the evolution equation for the relative entropy
\begin{align*}
  \frac{dE^*}{dt} &+mn\int_\Omega\r^{m-1} (\r-\r^*)^{n-2}|\nabla \r|^2 dx +
  \frac{1}{\delta}\|c-c^*\|^2_{H^1(\Omega)}\\
  &= n\int_\Omega \r(\r-\r^*)^{n-2}\nabla c\cdot \nabla \r dx - n\int_\Omega \r^{n-1}
  \nabla \r \cdot \nabla c dx + \frac{1}{\delta}\int_\Omega (\r-\r^*)(c-c^*) dx.
\end{align*}
We see that for general values $n>1$, also the entropy dissipation terms are not necessarily nonnegative. Moreover, the chemotactic drift term and the term arising from the cross-diffusion perturbation cancel out only if $n=2$.
These comments motivate us to consider the case $n=2$ only.
We recall that, due to \eqref{conditions}, $n=2$ implies that $1\leq m\leq 3$.
Thus we restrict us to the special case
\[
  m=1, \quad n=2,
\]
for which the global existence of a weak solution is guaranteed.
This choice allows us to show the decay of the solution to the
homogeneous steady state for certain values of $\delta$. For the
fully parabolic system we need additionally a smoothness
assumption on the solution, since the weak solution, obtained in
Theorem \ref{thm.ex}, cannot be used as a test function (in contrast to the
parabolic-elliptic system), see Remark \ref{rem.weak}. Notice that
we cannot start from the regularized problem, since there the
mass of $\r$ is not conserved, hence the system does not possesses a
constant homogeneous steady state.


Now, we are in the position to prove Proposition \ref{prop.decay}
(see the introduction).

\begin{proof}
\emph{Step 1: Decay of the relative entropy.}
We wish to employ $(2(\r-\r^*), (c-c^*)/2\delta)$ as test function in
\eqref{decay.rho}-\eqref{decay.c}, which is allowed if $\alpha=0$.
On the other hand, when $\alpha=1$, we need the smoothness assumption.

Since $n=2$, the evolution equation for the relative entropy
\[
  E^*(t)=\int_{\Omega} \left[(\r-\r^*)^2 +\alpha \frac{(c-c^*)^2}{2\delta}\right]dx
\]
reduces to
$$
  \frac{dE^*}{dt}+2\|\nabla \r\|_{L^2(\Omega)}^2
  + \frac{1}{\delta}\|c-c^*\|^2_{H^1(\Omega)}
  =\frac{1}{\delta}\int_\Omega (\r-\r^*)(c-c^*) dx.
$$
The difference to the entropy estimate in the existence proof is
that we can now apply the Poincar\'e inequality to $\r-\r^*$ to derive
the decay of the relative entropy for certain values of $\delta:$
\begin{align*}
  \frac{1}{\delta}\int_\Omega (\r-\r^*)(c-c^*) dx
  &\leq \frac{1}{\delta}\|\r-\r^*\|_{L^2(\Omega)}\|c-c^*\|_{L^2(\Omega)}
  \leq \frac{1}{2\delta}\|\r-\r^*\|^2_{L^2(\Omega)}
  + \frac{1}{2\delta}\|c-c^*\|^2_{L^2(\Omega)}\\
  &\leq \frac{C_P^2}{2\delta}\|\nabla\r\|^2_{L^2(\Omega)}
  + \frac{1}{2\delta}\|c-c^*\|^2_{L^2(\Omega)}.
\end{align*}
Finally, we obtain the entropy estimate
$$
  \frac{d}{dt}E^*(t)+\Big(2-\frac{C_P^2}{2\delta}\Big)
  \|\nabla \r\|^2_{L^2(\Omega)}+\frac{1}{2\delta}\|c-c^*\|^2_{H^1(\Omega)}
  \leq 0.
$$
We set $\kappa=\min\{1,4\delta-C_P^2\}/(4\delta)>0$. Then
\[
  \frac{d}{dt}E^* \leq -2\kappa E^*,
\]
implying the exponential decay of the relative entropy
\[
E^*(t)\leq E^*(0)e^{-2\kappa t}, \quad t>0.
\]
If $\alpha=1$ this immediately implies the desired decay of
$\r-\r^*$ and $c-c^*$ in $L^2(\Omega)$. For $\alpha=0$, the
relative entropy only gives the decay of $\r-\r^*$ in
$L^2(\Omega)$.

\emph{Step 2: Decay of $c-c^*$ for $\alpha=0$.}
Setting $v^*=c^*+\delta(\rho^*)^2$, we find that
\[
  v-v^* = c-c^*+\delta \big(\r^2-(\r^{*})^2\big)
  = c-c^*+2\delta \r^*(\r-\r^*)+\delta (\r-\r^*)^2.
\]
Replacing $c-c^*$ in the elliptic equation for $v$,
$$
  0 = \Delta v + \rho - c = \Delta v + (\rho-\rho^*) - (c-c^*),
$$
it follows that
\begin{equation}\label{ell.equ.v}
  0 = \Delta (v-v^*)+(1+2\delta \r^*)(\r-\r^*) + \delta (\r-\r^*)^2 -(v-v^*).
\end{equation}
Hence,
$$
  \|v-v^*\|^2_{H^1(\Omega)}
  = (1+2\delta \r^*)\int_\Omega (\r-\r^*)(v-v^*) dx
  +\delta \int_\Omega (\r-\r^{*})^2(v-v^*) dx.
$$
To determine the decay of $\|v-v^*\|_{H^1(\Omega)}$ we shall derive a
uniform bound for $v-v^*$. We prove in Step 3 below the
boundedness of $\r-\r^*$ in $L^\infty(0,\infty;L^4(\Omega))$.
Then elliptic regularity for \eqref{ell.equ.v} gives $v-v^*\in
L^\infty(0,\infty; H^2(\Omega))$ and, by Sobolev embedding,
$v-v^*\in L^\infty(0,\infty;L^{\infty}(\Omega))$. Hence,
proceeding with the above estimate and using Young's inequality,
we obtain
$$
  \|v-v^*\|^2_{H^1(\Omega)}
  \leq \frac12 \|v-v^*\|^2_{L^2(\Omega)} + C\|\r-\r^*\|^2_{L^2(\Omega)}
  + \delta\|v-v^*\|_{L^\infty(\Omega)}\|\r-\r^*\|^2_{L^2(\Omega)},
$$
such that
$$
  \|v-v^*\|_{H^1(\Omega)}^2
  \leq 2\big(C+\delta\|v-v^*\|_{L^\infty(\Omega)}\big) \|\r-\r^*\|^2_{L^2(\Omega)}
  \leq C e^{-2\kappa t}.
$$
This implies for the original unknown $c-c^*$ that
\[
  \|c-c^*\|_{L^1(\Omega)}
  \leq C\big(\|v-v^*\|_{L^2(\Omega)} + 2\delta\r^*\|\r-\r^*\|_{L^2(\Omega)}
  + \delta \|\r-\r^*\|^2_{L^2(\Omega)}\big)
  \leq Ce^{-\kappa t}.
\]

\emph{Step 3: Proof of $\r-\r^* \in L^\infty(0,\infty;L^4(\Omega))$.}
Using $4(\r-\r^*)^3$ as a test function in the first equation of \eqref{par-ell-v}
and employing the second equation in \eqref{par-ell-v}, we infer that
\begin{align*}
  \pa_t \int_\Omega & (\r-\r^*)^4 dx + 12 \int_\Omega (\r-\r^*)^2 |\nabla \r|^2 dx
  + 8\delta \int_\Omega \r^2 (\r-\r^*)^2 |\nabla \r|^2 dx \\
  &= 12 \int_\Omega \r(\r-\r^*)^2 \nabla \r\cdot \nabla v dx \\
  &= 3 \int_\Omega \nabla (\r-\r^*)^4\cdot \nabla v dx
  +  4\r^*\int_\Omega \nabla (\r-\r^*)^3\cdot\nabla v dx \\
  &= 4(1+2\delta\rho^*)\r^*\int_\Omega (\r-\r^*)^4 dx
  + (3+10\delta\r^*)\int_\Omega (\r-\r^*)^5 dx
  + 3\delta \int_\Omega (\r-\r^*)^6 dx \\
  &\phantom{xx}{}-  3\int_\Omega (v-v^*) (\r-\r^*)^4 dx
  - 4\r^*\int_\Omega(v-v^*)(\r-\r^*)^3 dx.
\end{align*}
The last two terms are bounded by using the nonnegativity of $\r$ and $v$ and
the Cauchy Schwarz inequality:
\begin{align*}
  -3\int_\Omega & (v-v^*) (\r-\r^*)^4 dx - 4\r^*\int_\Omega(v-v^*)\r(\r-\r^*)^2 dx
  + 4(\r^{*})^2\int_\Omega(v-v^*)(\r-\r^*)^2 dx \\
  &\leq  3v^* \int_\Omega (\r-\r^*)^4 dx
  + 4\r^*v^*\int_\Omega \r(\r-\r^*)^2 dx + 2(\r^{*})^2 \int_\Omega (v-v^*)^2 dx \\
  &\phantom{xx}{}+ 2(\r^{*})^2 \int_\Omega (\r-\r^*)^4 dx.
\end{align*}
Together with the estimate resulting from \eqref{ell.equ.v},
\[
  \|v-v^*\|^2_{H^1(\Omega)}
  \leq 2(1+2\delta \r^*)^2\int_\Omega(\r-\r^*)^2 dx
  + 2\delta^2 \int_\Omega (\r-\r^*)^4 dx,
\]
we obtain by interpolation:
\begin{align*}
  \pa_t \int_\Omega & (\r-\r^*)^4 dx + 12 \int_\Omega (\r-\r^*)^2 |\nabla \r|^2 dx
  + 8\delta \int_\Omega \r^2 (\r-\r^*)^2 |\nabla \r|^2 dx \\
  &\leq C \int_\Omega \left((\r-\r^*)^2 +  (\r-\r^*)^4+(\r-\r^*)^5 \right) dx
  + 3\delta \int_\Omega (\r-\r^*)^6 dx \\
  &\leq C(\delta) \int_\Omega (\r-\r^*)^2 dx
  + 4 \delta\int_\Omega(\r-\r^*)^6 dx.
\end{align*}
We already know the decay of $\|\r-\r^*\|_{L^2(\Omega)}$; hence, it remains to bound $\|\r-\r^*\|_{L^6(\Omega)}$ in terms of the entropy dissipation. To this aim, we use
the Gagliardo-Nirenberg with $\theta=d/(d+2)$ and the Young inequality:
\begin{align*}
  \|(\r & -\r^*)^3\|^2_{L^2(\Omega)}
  \leq C\big(\|\nabla (\r-\r^*)^3\|^{2\theta}_{L^2(\Omega)}
  \|(\r-\r^*)^3\|^{2(1-\theta)}_{L^1(\Omega)}
  + \|(\r-\r^*)^3\|^2_{L^1(\Omega)}\big) \\
  &\leq C\|\r(\r-\r^*)\nabla \r\|^{2}_{L^2(\Omega)}
  + C\|(\r-\r^*)\nabla \r\|^2_{L^2(\Omega)}
  + C \left(\int_\Omega |\r-\r^*|^3 dx\right)^2 \\
  &\leq C\|\r(\r-\r^*)\nabla \r\|^{2}_{L^2(\Omega)}
  + 6\|(\r-\r^*)\nabla \r\|^2_{L^2(\Omega)}
  + C\int_\Omega (\r-\r^*)^2 dx\int_\Omega (\r-\r^*)^4 dx.
\end{align*}
The decay $\|\r-\r^*\|_{L^2(\Omega)}\leq Ce^{-\kappa t}$ implies that
\begin{align*}
  \pa_t \int_\Omega & (\r-\r^*)^4 dx
  + 6 \int_\Omega (\r-\r^*)^2 |\nabla \r|^2 dx
  + 4\delta \int_\Omega \r^2 (\r-\r^*)^2 |\nabla \r|^2 dx \\
  &\leq Ce^{-2\kappa t} + Ce^{-2\kappa t}\int_\Omega (\r-\r^*)^4 dx,
\end{align*}
and we conclude that $\r-\r^* \in L^\infty(0,\infty;L^4(\Omega))$,
which completes the proof.
\end{proof}

\section{Numerical Simulations}\label{sec.numer}

This section is intended to illustrate numerically the solutions to the
fully parabolic system in two and three
space dimensions. We compare the results obtained for $\delta=0$ and
$\delta=0.005$ with various values for the exponent $n$ in \eqref{1.c}.
The simulations were carried out using the COMSOL Multiphysics package with
quadratic finite elements. The numerical solutions are for
illustration only; a more detailed comparison is the subject of
future work. We choose $\Omega=B_1(0)$ for simplicity.

\subsection*{The two-dimensional case}

We consider the fast-diffusion case $m=\frac12$ and prescribe the initial data
$$
  \r_0(x,y) = 80 (x^2+y^2-1)^2(x-0.1)^2+5, \quad c_0(x,y) = 0
  \quad\mbox{for }(x,y)\in B_1(0)
$$
with $M=\int_{\Omega} \r_0(x,y) dx = 25\pi/3 > 8\pi$ (see Figure \ref{fig.ic} left).
The maximal density is $\rho_{\rm max}=21.5$. We recall that solutions
to the classical parabolic-elliptic Keller-Segel model ($m=1$ and $\delta=0$)
blow up in finite time when the initial mass $M$ is sufficiently large.
More precisely, in the radial case, under an additional assumption on the
second moment, the solution blows up if $M>8\pi$ \cite{NSY97} or,
in the non-radial case, if $M>4\pi$ \cite{Nag01}.
Since $m>1$ leads to global existence results for the parabolic-elliptic system
\cite{CaCa06,Kow05}, one may conjecture that the cell density of the
parabolic-parabolic model blows up in finite time for $M>8\pi$ if $m<1$. We confirm
this conjecture numerically for the case $m=\frac12$ and the above initial datum.

\begin{figure}[ht]
\centering
\subfigure[$t=0$, $\r_{\rm max}=21.5$.]{
\includegraphics[width=63mm,height=58mm]{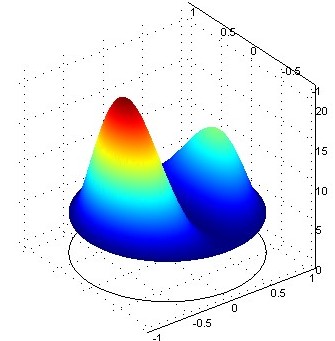}}
\subfigure[$t\approx 0.15$, $\r_{\rm max}\approx 1212$.]{
\includegraphics[width=63mm,height=58mm]{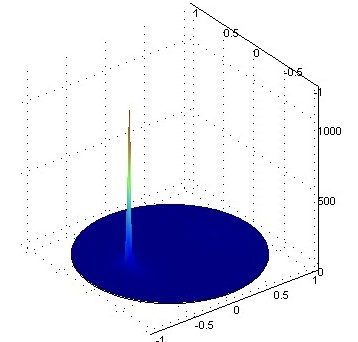}}
\caption{Cell density with $\delta=0$.}
\label{fig.ic}
\end{figure}

The nonlinear diffusion terms cause numerical difficulties whenever the
solution becomes close to zero. Indeed, the numerical solution may become
negative, and the simulations break down. Clearly, this can be handled by developing
a positivity-preserving numerical scheme, similarly as for the porous-medium
equation. Since we are using the black-box solver COMSOL Multiphysics,
we solve this problem simply by a projection method, i.e., we replace
diffusion terms by $\Delta(\max\{\rho,\eps\})^m$ with $m=\frac12$, $\eps=0.005$
and $\Delta(\max\{\rho,0\})^n$, respectively.

The cell density of the Keller-Segel model with $m=\frac12$ and $\delta=0$
at time $t\approx 0.15$ is depicted in Figure \ref{fig.ic} (right). Shortly after
that time, the simulations break down which indicates the blow up of solutions.
Surprisingly, the singularity forms in the interior of the domain in contrast
to the classical Keller-Segel model ($m=1$ and $\delta=0$) for which blow up
occurs at the boundary. Our numerical experiments confirm this behavior
for the model with $m=1$ and $\delta=0$ (results are not shown). Thus,
the unexpected behavior seems to be an effect of the fast cell diffusion.

Next, we turn to the case $m=\frac12$ and $\delta>0$. Figure \ref{fig.nd2}
shows the cell density at time $t=1000$ for various exponents
$n=\frac54,\frac{11}{8},\frac32,2$. The solutions have essentially reached
their steady state at $t=1000$. Notice that, according to \eqref{conditions},
the admissible parameter range for $n$ is $\frac54<n\le \frac32$.
Although some of the values for $n$ used in the simulations are theoretically
not admissible,
the solution exists numerically for all time. However, we observed numerical
difficulties for large values for $n$ (e.g.\ $n=6$) which may indicate that
the upper bound for $n$ in terms of $m$ is more than just a technical
assumption. We see from Figure \ref{fig.nd2} that the larger the value of $n$,
the more regular the solution becomes, at least in the tested parameter range
for $n$.

\begin{figure}[ht]
\centering
\subfigure[$n=\frac54$, $\rho_{\rm max}\approx 304$.]{
\includegraphics[width=63mm,height=58mm]{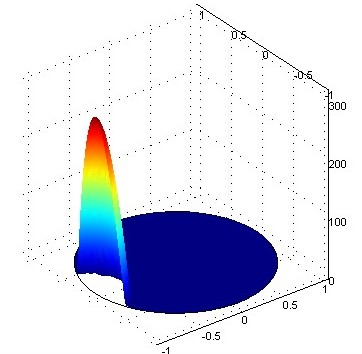}}
\subfigure[$n={11}{8}$, $\rho_{\rm max}\approx 174$.]{
\includegraphics[width=63mm,height=58mm]{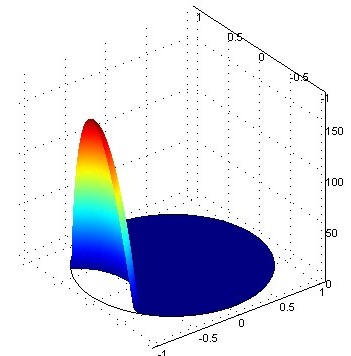}}
\subfigure[$n=\frac32$, $\rho_{\rm max}\approx 110$.]{
\includegraphics[width=63mm,height=58mm]{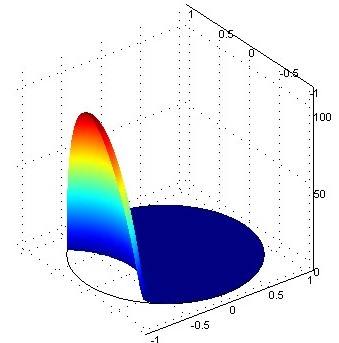}}
\subfigure[$n=2$, $\rho_{\rm max}\approx 30$.]{
\includegraphics[width=63mm,height=58mm]{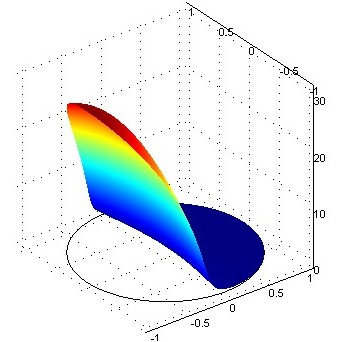}}
\caption{Cell density at time $t=1000$ with $\delta=0.005$.}
\label{fig.nd2}
\end{figure}

In the limit of vanishing additional cross-diffusion $\delta\to 0$, we expect
that the solutions converge to the solution to the corresponding Keller-Segel
model with $\delta=0$. This is numerically confirmed in Figure \ref{fig.ded2}.
For $\delta=0.005$, the cell density reaches its maximum at the boundary
(see Figure \ref{fig.nd2}), whereas the maximum is attained in the interior
of the domain for very small values of $\delta$. Thus,
it seems that the cross-diffusion regularization produces a stationary state
which is more concentrated on the boundary.

\begin{figure}[ht]
\centering
\subfigure[$\delta=10^{-4}$, $\rho_{\rm max}\approx 107$.]{
\includegraphics[width=63mm,height=58mm]{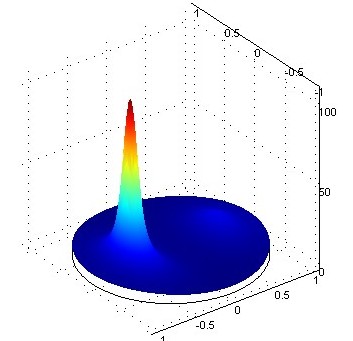}}
\subfigure[$\delta=10^{-6}$, $\rho_{\rm max}\approx 681$.]{
\includegraphics[width=63mm,height=58mm]{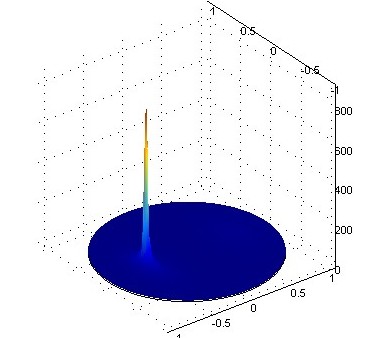}}
\subfigure[$\delta=10^{-8}$, $\rho_{\rm max}\approx 1201$.]{
\includegraphics[width=63mm,height=58mm]{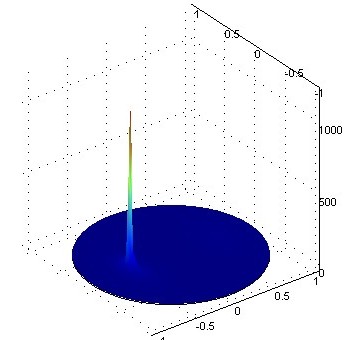}}
\subfigure[$\delta=0$, $\rho_{\rm max}\approx 1212$.]{
\includegraphics[width=63mm,height=58mm]{d0-0.jpg}}
\caption{Cell density at time $t\approx 0.15$ with $n=3/2$.}
\label{fig.ded2}
\end{figure}


\subsection*{The three-dimensional case}

We consider the linear case $m=1$ and prescribe the initial data
$$
  \r_0(x,y,z) = 10+80(x^2+y^2+z^2-1)^2(x-0.4)^2,
  \quad c_0(x,y,z) = 0, \quad (x,y)\in B_1(0),
$$
see Figure \ref{fig.ic2} (left). In three space dimenions, even for the
parabolic-elliptic system, there is no
critical threshold known for the occurence of blow up. A more
complicated functional relation between the second moment and the
$L^{3/2}$ norm of the cell density has been derived in \cite{CPZ04} as a sufficient
condition for finite-time blow up. We observed
that the numerical solution to the fully parabolic model with $\delta=0$
breaks down after time $t\approx 0.46$, which may indicate a blow-up formation
(see Figure \ref{fig.ic2} right).

The initial data in Figure \ref{fig.ic2} is
represented using slices, since the highest values occur inside
the domain, whereas for the following simulations, we use the
level-set representation which is more appropriate for
demonstrating aggregation phenomena on the boundary.

\begin{figure}[ht]
\centering
\subfigure[$t=0$, $\r_{\rm max}\approx 46$.]{
\includegraphics[width=63mm,height=58mm]{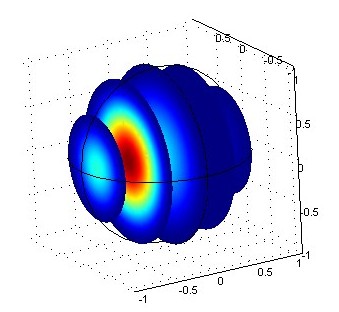}}
\subfigure[$t\approx 0.46$, $\r_{\rm max}\approx 5460$.]{
\includegraphics[width=63mm,height=58mm]{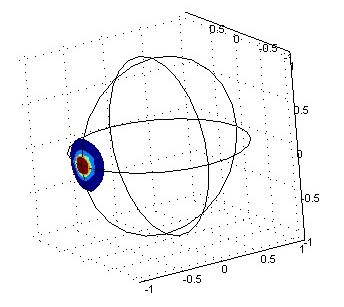}}
\caption{Cell density with $\delta=0$.}
\label{fig.ic2}
\end{figure}

In Figure \ref{fig.n3d} we compare the cell density at time
$t=1000$ with $\delta=0.005$ and
$n=\frac65,\frac{13}{10},\frac32,\frac74$. At $t=1000$, the
solutions have essentially reached the (non-homogeneous) steady
state. As we already proved in Proposition \ref{prop.decay}, when
performing the same simulation with $n=2$, the solution converges
to the homogeneous steady state. As in the two-dimensional
situation, the maximal cell density is achieved at the boundary.

\begin{figure}[ht]
\centering
\subfigure[$n=\frac65$, $\rho_{\rm max}\approx 1023$.]{
\includegraphics[width=63mm,height=58mm]{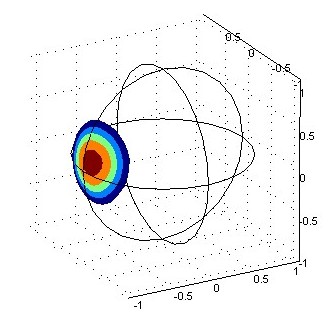}}
\subfigure[$n=\frac{13}{10}$, $\rho_{\rm max}\approx 487$.]{
\includegraphics[width=63mm,height=58mm]{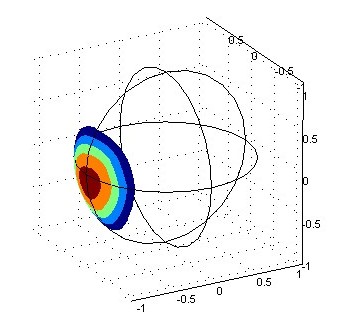}}
\subfigure[$n=\frac32$, $\rho_{\rm max}\approx 158$.]{
\includegraphics[width=63mm,height=58mm]{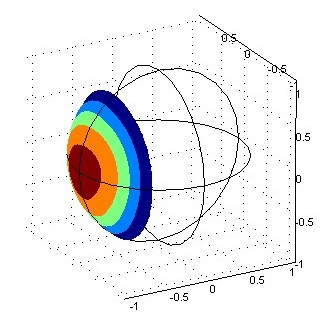}}
\subfigure[$n=\frac74$, $\rho_{\rm max}\approx 58$.]{
\includegraphics[width=63mm,height=58mm]{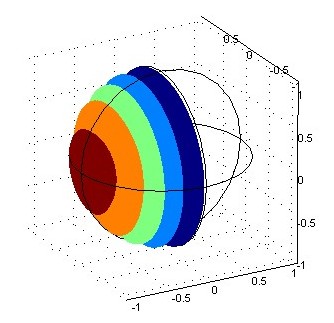}}
\caption{Cell density at time $t=1000$ with $\delta=0.005$.}
\label{fig.n3d}
\end{figure}


\end{document}